\documentclass{amsart}
\usepackage{latexsym}
\usepackage{amssymb}
\usepackage{amsmath}
\usepackage{enumerate}
\usepackage[dvips]{graphicx}
\usepackage[latin1]{inputenc}
\title{Perturbing Misiurewicz parameters in the exponential family}
\author{Neil Dobbs}

\newcommand\cQ{\mathcal{Q}}

\newcommand\cD{\mathcal{D}}

\newcommand\cJ{\mathcal{J}}

\newcommand\cL{\mathcal{L}}
\newcommand\cR{\mathcal{R}}

\newcommand\cW{\mathcal{W}}

\newcommand\re{{\Re}}

\newcommand\eps{\varepsilon}
\DeclareMathOperator{\dist}{dist}
\DeclareMathOperator{\diam}{diam}

\newcommand\ccc{\mathbb{C}}

\newcommand\Z{\mathbb{Z}}

\newcommand\N{\mathbb{N}}

\newtheorem*{mainthm}{Main Theorem}

\newtheorem{conj}{Conjecture}
\newtheorem{thm}{Theorem}

\newtheorem{lem}[thm]{Lemma}
\newtheorem{prop}[thm]{Proposition}

\begin{document}
\date{\today}
\begin{abstract}
  In one-dimensional real and complex dynamics, a map whose post-singular (or post-critical) set is bounded and uniformly repelling is often called a Misiurewicz map. In results hitherto, perturbing a Misiurewicz map is likely to give a non-hyperbolic map, as per Jakobson's Theorem for unimodal interval maps. This is despite genericity of hyperbolic parameters (at least in the interval setting). We show the contrary holds in the complex exponential family $z \mapsto \lambda \exp(z)$: Misiurewicz maps are Lebesgue density points for hyperbolic parameters. 
  As a by-product, we also show that Lyapunov exponents almost never exist for exponential Misiurewicz maps. The lower Lyapunov exponent is $-\infty$ almost everywhere. The upper Lyapunov exponent is non-negative and depends on the choice of metric. 
\end{abstract}
    
\maketitle

\section{Introduction}
Jakobson's Theorem (\cite{Jakobson:Quad}) from 1981 is one of the more celebrated and striking results in dynamical systems. In the real quadratic (or logistic) family $f_a: x \mapsto ax(1-x)$, Jakobson showed that there is a positive measure set of parameters $a$ close to the Chebyshev parameter $a=4$ for which the map has an absolutely continuous, $f_a$-invariant probability measure $\mu_a$. One can contrast this with the result (\cite{GS:Fatou, Lyubich:QP12}), due to Graczyk and \'Swi\c atek and to Lyubich, which states that the set of hyperbolic parameters is open and dense, to emphasise the intricacy of quadratic dynamics.
    Rees in \cite{Rees:Rat} generalised Jakobson's result to rational maps of the Riemann sphere. Benedicks and Carleson extended these results to the H\'enon family in \cite{BC:Henon}. In these settings, one starts with a map with a repelling post-critical set, and sufficiently small perturbations are likely to give non-hyperbolic parameters. 
    In this paper we present a counter-example to this paradigm  in the complex exponential family. 

    In the exponential family $f_\lambda : z \mapsto \lambda e^z$, a parameter $\lambda$ is called a \emph{Misiurewicz} parameter if $\overline{\{f_\lambda^n(0): n\geq0\} } \subset \ccc$ is a bounded, hyperbolic repelling set. The simplest example is for $\lambda = 2\pi i$. 
    For Misiurewicz parameters, the Julia set is the entire complex plane (or, regarding $f$ as a meromorphic map, the Julia set is the entire Riemann sphere). In particular, there are dense orbits. 

    A parameter $\lambda$ is called \emph{hyperbolic} if $f_\lambda$ has an attracting periodic orbit. For hyperbolic $\lambda$, almost every orbit is in the basin of attraction of the attracting periodic orbit. Any $\lambda$ with $|\lambda| < 1/e$ is hyperbolic. 
    \begin{mainthm}
        In the complex exponential family, Misiurewicz parameters are Lebesgue density points for the set of hyperbolic parameters. 
    \end{mainthm}
    By this we mean, if $\lambda_0$ is a Misiurewicz parameter, $H$ is the set of hyperbolic parameters and $m$ denotes Lebesgue measure, then 
    $$
    \lim_{r\to 0^+} \frac{m(B(\lambda_0,r)\cap H)}{m(B(\lambda_0,r))} =1.
    $$

    For Misiurewicz parameters in the exponential family, there is a conservative, $\sigma$-finite, ergodic, absolutely-continuous invariant measure. It even has a real-analytic density off the post-singular set (\cite{Me:Philsoc}).  However, it was shown in  \cite{MeBartek, KotSwi:Us} that no absolutely continuous invariant probability measure can exist. 
    To prove the main theorem, strong estimates on the dynamics of Misiurewicz maps are required. 
    The same estimates, with only a slight extension, permit one to show that for Misiurewicz maps, the Lyapunov exponent of a point exists almost nowhere. We use $Df$ to denote the derivative of $f$ with respect to the Euclidean metric. 
    \begin{thm} \label{thm:LE1}
        Let $f$ be a Misiurewicz map from the exponential family. For Lebesgue almost every $z \in \ccc$, 
        $$
    \liminf_{n \to \infty} \frac{1}{n} \log |Df^n(z)| = -\infty,
        $$
        while
        \begin{equation}\label{eq:tle1}
    \limsup_{n \to \infty} \frac{1}{n} \log |Df^n(z)| = +\infty.
    \end{equation}
\end{thm}
However, the plane is not compact and there is a choice of Riemannian metric. For any metric $\rho$, 
    let $D_\rho g$ denote the derivative of $g$ with respect to $\rho$. 
    In particular,
    for the spherical metric $\sigma$,
$$
D_\sigma g(z) := \frac{1+|z|^2}{1+ |g(z)|^2} Dg(z).$$
    \begin{thm} \label{thm:LE2}
        Let $f$ be a Misiurewicz map from the exponential family. For Lebesgue almost every $z \in \ccc$ and every Riemannian metric $\rho$,
        \begin{equation}\label{eq:tle2}
    \liminf_{n \to \infty} \frac{1}{n} \log |D_\rho f^n(z)| = -\infty
    \end{equation}
        and
        $$
    \limsup_{n \to \infty} \frac{1}{n} \log |D_\rho f^n(z)| \geq 0,$$
    while for the spherical metric $\sigma$,
        $$
    \limsup_{n \to \infty} \frac{1}{n} \log |D_\sigma f^n(z)| = 0.$$
\end{thm}
One could replace $\log$ in equation \eqref{eq:tle1} by any finite composition of logarithms and the result would still hold, though we do not quite show this (and similarly for \eqref{eq:tle2}, remembering to take absolute values); the number $4$ in Lemma~\ref{lem:QRtoL} was chosen rather arbitrarily. 

For a class of maps of the unit interval with negative Schwarzian derivative, Keller (\cite{Keller:ExpActHopf}) showed that if $\limsup_{n \to \infty} \frac1n \log |Df^n(x)|  > 0$ for almost every $x$, then there exists an absolutely continuous invariant probability measure. Theorem~\ref{thm:LE1} implies that the same does not hold generally in the exponential family, at least for the Euclidean metric. It would be interesting to know whether the following conjectures are equivalent. 
\begin{conj}
    Let $f : z \mapsto \lambda e^z$. 
    For the spherical metric $\sigma$ and Lebesgue almost every $z \in \ccc$, 
        $$
    \limsup_{n \to \infty} \frac{1}{n} \log |D_\sigma f^n(z)| = 0.$$
\end{conj}

\begin{conj}
    No map 
    from the exponential family admits an absolutely continuous invariant probability measure. 
\end{conj}


    \bigskip
    Misiurewicz parameters (and maps) have a long and involved history in the field of one-dimensional dynamics. Introduced by Misiurewicz in \cite{Mis:Mis} for smooth maps of the interval, they became the first examples where some non-trivial condition on the behaviour of critical orbits guaranteed the existence of absolutely continuous invariant probability measures. 
    This result was superseded by many more in interval dynamics, see \cite{BRSvS} for one of the latest and strongest. The concept of Misiurewicz parameter exists in other contexts too, see \cite{GPS:Mis, MisBen:Flat, Bad:Rare, MeBartek} for example. 
    The articles \cite{Jakobson:Quad, Thun:Flat, Rees:Rat, BC:Quad} 
    all find positive measure sets of non-hyperbolic parameters (indeed one admitting absolutely continuous invariant probability measures) in a neighbourhood of Misiurewicz parameters. On the other hand, Misiurewicz parameters have zero Lebesgue measure, in general (\cite{Sands:Rare, Asp:Rare, Bad:Rare}). 

    In \cite{Thun:Flat}, Thunberg finds positive measure sets of non-hyperbolic parameters in unimodal families of interval maps with critical points of type $\exp(-|x|^{-\alpha})$, provided $\alpha < 1/8$. We showed in \cite{Me:Cusp} that if $\alpha \geq 1$, no absolutely continuous invariant probability measure with positive entropy can exist, as was shown for Misiurewicz parameters in the same setting in \cite{MisBen:Flat}. 

    Structural instability of Misiurewicz parameters in the exponential family was shown in \cite{Makienko:Exp, UZ:ExpInstability, GKS:Nonrec}. For the (non-Misiurewicz) map $z \mapsto e^z$, the orbit of 0 is a (wild) metric attractor attracting almost every orbit (\cite{Rees86, Lyubich:Exp}), although generic orbits are dense. This map is a density point for hyperbolic maps in the exponential family \cite{wang2008most}. For those interested in the structure of parameter space of the exponential family (as opposed to metric properties), we refer to \cite{RempeSchl:Exp}.

    It has been suggested by  Hubbard that hyperbolic parameters should have full measure in the exponential family, see \cite{Qiu:Exp} (where it is shown that non-hyperbolic parameters have full Hausdorff dimension). This would be a stronger conjecture than density of hyperbolic parameters, and this paper and \cite{wang2008most} could be viewed as first small steps in that direction.  

    One could ask about the complex quadratic family $f_c : z \mapsto z^2 +c$ with Julia set $\cJ_c$. Rivera-Letelier (\cite{JuanRL:continuity}) showed that if $c \in \cJ_c$ and $c$ is non-recurrent, then $c$ is a density point for hyperbolic parameters (ones for which $c$ is in the basin of a periodic attractor, finite or at infinity). 
    Aspenberg in \cite{Asp:RatFund} extended this to result to more general rational maps for which the Julia set is not the whole sphere. In both these cases, basins of periodic attractors are open and dense in the sphere. It is natural in these cases to expect that, with expansion along the post-critical orbits, a small perturbation is likely to send the critical orbits into the attracting basins. 
    What is strange in the exponential setting is that Misiurewicz parameters are density points for hyperbolic parameters even though the Julia set at the Misiurewicz parameter is the whole space.

    \bigskip
    For a map $f_\lambda$ from the exponential family, $f_\lambda(z) = Df_\lambda(z)$ and $|f_\lambda(z)| = |\lambda|e^{\Re(z)}$, so $f_\lambda$ is $2\pi i$-periodic, $f_\lambda$ maps vertical lines to circles, horizontal lines to rays emanating from 0, and rectangles of height $2\pi$ onto annuli centred at 0. Points far to the left get mapped extremely close to 0, and points far to the right get mapped extremely far from 0. 

    In Appendix D of \cite{Milnor:TwoCrit}, Milnor shows that our choice of realisation of the exponential family is, in some sense, as good as any other: any entire map with asymptotic values at $\infty$ and at some finite point, and without critical points, is conjugate to a map from the exponential family. Alternative reasonable choices are $g_\kappa : z \mapsto e^z + \kappa$ and $g_\kappa : z \mapsto e^{\kappa z}$.  

         \section{Global definitions} \label{sec:glob}

    Throughout the paper, let $ f = f_{\lambda_0} : z \mapsto \lambda_0 \exp(z)$, for some Misiurewicz parameter $\lambda_0 \in \ccc$; in particular the post-singular set $$P(f) := \overline{\{f^n(0): n\geq 0\}  }$$
    is a bounded hyperbolic repelling set, so there are $n_0, \alpha > 0$ such that $|Df^{n_0}(z)| > \exp(2\alpha )$ for all $z \in P(f)$. 
    By continuity, we can fix $\varepsilon_0, \delta \in (0, \frac12 )$ 
             such that  for all $\lambda \in B(\lambda_0, \eps_0)$ and all  $z \in B(P(f), 3\delta)$, 
            $$|Df_\lambda^{n_0}(z)| > \exp(\alpha).$$ 
    Set $V := B(P(f), \delta)$.

    We shall denote by $\Delta > 1$ the modulus giving a  Koebe \emph{distortion bound} of 2, 
            that is, the minimal number such that for any 
            univalent map $g$  on $B(0, \Delta)$, the distortion of $g$ on $B(0,1)$ is bounded by $2$:   
            $$
            \sup_{y,z \in B(0,1)} \left|\frac{Dg(y)}{Dg(z)}\right| \leq 2.
            $$
            We shall repeatedly use the following fact. 
            \begin{lem} \label{lem:distnQ}
                For any simply-connected open set $U$ with $\dist(U, P(f)) > \Delta \diam(U)$, if $f^n(z) \in U$ then a neighbourhood of $z$ is mapped biholomorphically onto $U$ with distortion bounded by $2$. 
            \end{lem}
            \begin{proof}
                Since $P(f) \cap B(f^n(z), \Delta \diam(U))
                  = \emptyset$, there is a neighbourhood of $z$ mapped biholomorphically onto 
        $B(f^n(z), \Delta \diam(U))$. By definition of $\Delta$, a neighbourhood of $z$ gets mapped with distortion bounded by $2$ onto $B(f^n(z), \diam(U)) \supset U$, as required.
            \end{proof}

         The notation 
         $A(y; a_1,a_2)$  is used for the annulus 
         centred on $y \in \ccc$ with inner and outer radii of lengths $a_1, a_2$. 

        Denote by $\cR(x)$ the right half-plane 
        $$ \cR(x) := \{z \in \ccc : \Re(z) \geq x\}$$ and denote by $\cL(x)$ the left half-plane $\ccc \setminus \cR(x)$. Denote by $\cQ$ the collection of squares of the form $$\{z : 2k\pi \leq \Re(z) < (2k+2)\pi; 2j\pi \leq \Im(z) < (2j+2)\pi\},$$ for $j, k \in \Z$. Each square has diameter $2\sqrt{2} \pi < 9$.

	Further  definitions occur throughout the paper. These include constants $\delta_0 \in (0, \delta)$ and $N_1, M>0$ at the start of Section~\ref{sec:FE}; following Lemma~\ref{lem:K2},  constants $M_0, r_0$ and holomorphic motion $h$ with $h(0, \lambda) = a_K(\lambda - \lambda_0)^K$ to first order, with $a_K \in \ccc \setminus \{0\}$ and $K$ a positive integer; 
	just prior to Proposition~\ref{prop:xigrows}, $\xi_n : \lambda \mapsto f^n_\lambda(0)$. 

    \section{Structure of the proof}
    The following simple lemma guides the proof of the Main Theorem. 
    \begin{lem} \label{lem:guide}
       For each $C>0$, for all $x$ large enough, the following holds. 
       Let $\lambda \in \ccc$ satisfy $|\lambda| < x$, let $n \geq 1$ and  
       suppose that  the following holds. 
       \begin{enumerate}
           \item
       $f_\lambda^n$ maps a neighbourhood of $0$ biholomorphically onto $B(f^n_\lambda(0), 1)$;
   \item \label{en:xv}
         $B(f^n_\lambda(0), 1) \subset \cL(-e^{x + \sqrt{x}} + 3\pi)$;
     \item \label{en:derhyp}
                $|Df_\lambda^n(0)| < \exp(Ce^{x}) |\re(f_\lambda^n(0))|^4$.
        \end{enumerate}
        Then $f_\lambda$ has a hyperbolic attracting periodic orbit.
   \end{lem}
   \begin{proof}
       Set $v = f^n_\lambda(0)$ for the sake of readability. If $x$ is large, so is $-\re(v)$, by \ref{en:xv}. Thus
       \begin{equation} \label{eq:guide1}
       \exp(\re(v)/2) |\re(v)|^4 <1.
       \end{equation}
       Since $\re(v) < -e^{x+\sqrt{x}} + 3\pi$, 
       \begin{equation} \label{eq:guide2}
       \exp(\re(v)/2) ex 2\Delta \exp(C e^x) <1.
       \end{equation}
       
       Using a (Koebe) distortion bound of $2$, $f^n_\lambda$ maps $B_\lambda := B(0, 1/2\Delta |Df^n_\lambda(0)|)$ into $B(v, 1/\Delta)$. 
       Thus $f_\lambda^{n+1}(B_\lambda)$ is contained in $B(0,  r)$, with $r = e |\lambda| \exp(\re(v))$. 
           Thanks to \ref{en:derhyp} and the bound $|\lambda| < x$, 
	   $$r 2\Delta |Df^n_\lambda(0)| < e x \exp(\re(v)) 2\Delta \exp(C e^x) \re(v)^4 < 1,$$
	   the latter inequality obtained combining \eqref{eq:guide1} and \eqref{eq:guide2}.
            Hence $\overline{B(0,r)} \subset B_\lambda$.

       Since  
       $f^{n+1}_\lambda(B_\lambda) \subset B(0,r)$ and 
       $\overline{B(0,r)} \subset B_\lambda$, $f_\lambda$ has a hyperbolic attracting periodic orbit. 
   \end{proof}
   With expansion along the post-singular orbit of $f$, one can often transfer estimates for large sets of points in phase space for $f$ into estimates on the post-singular orbit of $f_\lambda$ for large sets of parameters $\lambda$. 
   It is natural, therefore, to try to find, for $f$, a large set of points which enter the left half-plane with estimates related to those of the lemma. A substantial portion of the paper comprises of this effort. 

   Note that as $\lambda$ approaches $\lambda_0$, there will be a huge build-up or derivative initially as $f_\lambda^j(0)$ spends a long time near $P(f)$. To counteract this, one needs eventually to land, eventually, far off to the left, before getting mapped extremely close to $0$ to cancel out the derivative build-up. 

   A general result (\cite{GKS:Nonrec}) implies that for exponential Misiurewicz maps (amongst others), every forward-invariant compact set is hyperbolic repelling (but with no rate estimate). This in turn implies that the measure of the set of points remaining in any bounded set for $n$ iterates is exponentially small in $n$. However, to deal with smaller parameter perturbations (for $\lambda$ closer to $\lambda_0$),  we need to find large sets of points going ever further to the left. Thus we need to know how the exponential rate depends on the size of the bounded set. 
   In Proposition~\ref{prop:hyp}, we obtain the relevant hyperbolicity estimates. 
   For an exponential Misiurewicz map the derivative grows exponentially fast, except when it is slowed by the occasional passage close to zero. 
   
   Lemma~\ref{lem:allballs} is useful and curiously does not hold for quadratic maps, say. The lemma implies a distortion bound, which together with the derivative growth estimates, allows one to relate large and small scales and hence, via a porosity-type argument, to estimate how long it takes for a large proportion of points to make a first entry into a right half-plane $\cR(x)$. Boot-strapping, we show that for most of these points, the first entry actually lands in $\cR(x + 2\sqrt{x})$. In Section~\ref{sec:FR} we study the dynamics of points in a far-right half-plane, showing that most points go further and further to the right before eventually landing far out to the left.  
   
   These results are gathered together in Proposition~\ref{lem:s0}, which says that if you start from a reasonable, reasonably large set close to $P(f)$, then most points in that set first enter a far-left half-plane in a bounded amount of time and with a derivative bound. Points may land far, far to the left, so the derivative bound depends not only on the half-plane but also on the real part of the landing location. 
   
   With the necessary ingredients in place, we pause to prove Theorems~\ref{thm:LE1} and~\ref{thm:LE2} in Section~\ref{sec:Lyap}. Only the estimate for the spherical derivative is a little complicated. 
   
    Returning to the proof of the main theorem,  we begin our parameter-based estimates. We show that points from Proposition~\ref{lem:s0} do not move too, too fast as the parameter moves. The continuation of a point is defined to have a similar orbit and the same first entry to the left half-plane. The estimates depend on the parameters considered being in a tiny ball, the time being bounded and the orbit being a certain distance away from $0$. 
    
    Going backwards from a large neighbourhood of a point in $P(f)$ to a tiny neighbourhood of $0$ (or of $\lambda_0$) is more delicate, though at this stage the arguments are well-understood. The estimates are also a little less cumbersome going forwards than backwards. For completeness, and because the desired estimates are not simple to extract from \cite{Bad:Rare} (itself based on \cite{Asp:Rare}), we include proofs of the estimates. 
    Sectors of small annuli centred on $\lambda_0$ in parameter space get mapped biholomorphically and with bounded distortion onto reasonably large sets near $P(f)$ by the map $\lambda \mapsto f^n_\lambda(0)$, for some $n$ depending on the annulus, see Lemma~\ref{lem:xiinject}. 

	With the various (parametric) derivative estimates, it is not too hard then, in Proposition~\ref{prop:main}, to match up most parameters with orbits and, using Lemma~\ref{lem:guide}, to show that for these parameters the maps are hyperbolic. 

\bigskip
    
A note of comparison, Wang and Zhang (\cite{wang2008most}) showed that $1$ is a density point for hyperbolic parameters. For $g: z \mapsto e^z$, most points near $1$ follow the orbit of $1$ out towards infinity until escaping a small neighbourhood of the orbit, then take more steps towards $+\infty$, then get mapped extremely close to $-\infty$ and then super-close to $0$ with derivative close to $0$. One only needs to study the dynamics close to infinity and along the orbit of $1$, so the arguments are relatively elegant and straightforward. Moreover, one can calculate by hand that the derivative of $\lambda \mapsto f_\lambda^n(0)$ is positive (and increasing in $n$) at $1$, so in particular it is non-zero. If this were known to be the case for Misiurewicz parameters, one would have  $K=1$ in equation \eqref{eq:hlam} and one could deal with balls instead of annuli, a minor simplification. 
One could attempt the parameter exclusion method as per \cite{wang2008most} in the current setting, and it would work as long as one remains close to $P(f)$, though something along the lines of Proposition~\ref{prop:xigrows} would of course still need to be shown. However, continuing on beyond the comfort of a neighbourhood of $P(f)$, where one has injectivity and distortion control, would likely lead to many sleepless nights.

\section{Non-uniform hyperbolicity}

In this section we gather some estimates on the growth of the derivative along individual orbits and their neighbourhoods. In the following section we will use these estimates to compare small and large scales and derive some measure estimates. Recall that $f = f_{\lambda_0}$ and $\lambda_0$ is a Misiurewicz parameter. 

    \begin{lem}
            The Julia set is $\overline{\ccc}$ and $|\lambda_0| \geq 1/e$. 
        \end{lem}
        \begin{proof}
            The first statement follows immediately from Theorems~3-5 of \cite{EremenkoLyubich:Entire}, since the post-singular set is uniformly repelling. 
            Were $|\lambda_0| < 1/e$, then  $\overline{f(B(0,1))} \subset B(0,1)$ and $f$ would have an attracting fixed point.  
        \end{proof}
        \begin{lem} \label{fact:nice}
        For each $z$ in $\ccc \setminus P(f)$, there are arbitrarily small neighbourhoods $U_z$ on which the first return map $\phi$ to $U_z$ is expanding (that is, $|D\phi| > \gamma_z > 1$). 
    \end{lem}
        \begin{proof}
            This is part (iii) of \cite[Lemma~11]{Me:Philsoc}, knowing that the Julia set is $\ccc$.
            \end{proof}

    \begin{lem} \label{lem:gammadfk}
            Given any $\theta >0$, there is a $\beta \in (0,1)$ such that, for any $z \in \ccc$ and $k \geq 0$, if $\dist(f^k(z), P(f)) \geq \theta$ then $|Df^k(z)| > \beta$. 
        \end{lem}
        \begin{proof} 
            By Lemma~\ref{lem:distnQ}, some neighbourhood $W$ of $z$ is mapped biholomorphically onto $B(f^k(z), \theta/\Delta)$ with distortion bounded by  $2$. But $f$ is not univalent on any ball of radius $\pi$, so $W$ cannot strictly contain a ball of radius $\pi$. Combining these two facts, the derivative of $f^k$ on $W$ cannot be too small. 
            \end{proof}
            
            \begin{lem} \label{lem:Mdev}
                There is an $M>3$ such that, for all $z \in \ccc$ and $k \geq 1$, if $|f^k(z)| \geq M$ then $|Df^k(z)| > 3$. 
            \end{lem}
        \begin{proof} 
            Let $\theta >0$ and let $\beta$ be given by  Lemma~\ref{lem:gammadfk}, 
            Take $M > 3/\beta$ sufficiently large that $f(B(P(f),\theta)) \subset B(0,M)$. If $f^k(z)\geq M$ then $\dist(f^{k-1}(z), P(f)) \geq \theta$, so $|Df^{k-1}(z)| > \beta$. But $|Df^k(z)| = |f^k(z)||Df^{k-1}(z)| \geq M\beta > 3$. 
            \end{proof}
            \begin{lem} \label{lem:Mdev2}
                Given $M_1 >0$ there is an $M_2>0$ such that, for all $z \in \ccc$ and $k \geq 2$, if $|f^k(z)| \geq M_2$ then $|Df^k(z)| > M_1 |f^k(z)|$. 
            \end{lem}
        \begin{proof} 
            Take $M_2$ large enough that $|f^{k-1}(z)|$ must be larger than $M_1$ and $|f^{k-2}(z)|$ must be larger than $M$, where $M$ comes from Lemma~\ref{lem:Mdev}.
            \end{proof}


    \begin{lem} \label{lem:mindev}
        There is some $\beta_1 >0$ such that, for each $z \in \ccc$ and $k \geq 1$, 
        \begin{equation} \label{eq:mindeveq}
            |Df^k(z)| \geq \beta_1 \inf_{1 \leq j \leq k} |f^j(z)|.\end{equation}
    \end{lem}
        \begin{proof}
            Let $n\leq k$ be maximal such that $|Df^n(z)| \geq 1$. 
            If 
            $n \geq 1$ then 
            \begin{equation*} 
                |Df^n(z)| 
            = \prod_{1\leq j \leq n} |f^j(z)| 
                \geq \inf_{1 \leq j \leq n} |f^j(z)|. 
            \end{equation*}
            Note that if $n = k$, \eqref{eq:mindeveq} holds with $\beta_1 =1$. Assume now that $n < k$.  
            Let $M$ be given by Lemma~\ref{lem:Mdev}, so $|f^j(z)| < M$ for $j = n+1, \ldots k$. 
             By the chain rule and choice of $n$, $|Df^k(z)| \geq |Df^{k-n}(f^n(z))|.$  
             It now suffices to prove that 
            $$|Df^{k-n}(f^n(z))| \geq \beta_1 \inf_{n+1 \leq j \leq k} |f^j(z)|. 
            $$
            Rewriting, it suffices to prove the lemma under the assumption $|f^j(z)| < M$ for $j = 1, \ldots, k$. 

            Recall that $V$ is globally defined in Section~\ref{sec:glob} as a small neighbourhood of the postsingular set. We can divide the orbit into three pieces: a first stretch which ends outside $V$,  a final stretch spent entirely inside $V$, and in between a single iterate. So, 
            let $n \leq k$ be maximal such that $f^n(z) \notin V$, if it exists, otherwise set $n=0$. 
            Let $\beta \in (0,1)$ be given by Lemma~\ref{lem:gammadfk}.  
            Then 
            $$|Df^n(z)| \geq \beta \geq \beta |f^n(z)|/M,$$
            by assumption. If $n=k$, we are done (provided $\beta_1 \leq \beta/M$). So assume $n <k$. 
            Now $f^{n+1}(z), \ldots, f^k(z) \in V$, by definition of $n$. Since $|Df^{n_0}| > 1$ on $V$, it follows that $|Df^{k-n-1}(f^{n+1}(z))|$ is bounded below by the constant $$\beta_2 := \min_{1\leq j<n_0}\inf_{y\in V}{|Df^j(y)|} >0.$$
            Then $|Df^k(z)| \geq \beta \beta_2 |Df(f^n(z))|$. Taking $\beta_1 := \min (\beta/M, \beta \beta_2 \beta_3)$ works, where $\beta_3 = |\lambda_0|e^{-M}$.
        \end{proof}

        In the following proposition, we use exponential growth when one remains in a neighbourhood of $P(f)$, exponential growth when one remains in a bounded region disjoint from that neighbourhood, plus absolute growth if an iterate lands outside a large bounded region, to give some sort of non-uniform hyperbolicity statement for Misiurewicz maps. 
        
    \begin{prop} \label{prop:hyp}
        There are $N, N_1 >0$  such that for each $z$ there is a $j \leq N+ N  |\log|f(z)||$ with $|Df^j(z)| > 3$ and $|Df^i(z)|,|f^i(z)| \leq N_1 + |f(z)| $ for $i=1, \ldots, j$. 
    \end{prop}

    \begin{proof}
        We can assume $|f(z)| \leq 3$, otherwise one can simply take $j=1$.
        Set $p := 1+ n_0 \lceil (2 -\log|f(z)|  )/\alpha \rceil$,   and
        note that $p$ is bounded by an affine function of $|\log |f(z)||$.  
        Let $k \geq 1$ be minimal such that $f^k(z) \notin V$. 
        If $k \geq p$, 
        $$|Df^{p}(z)| \geq |f(z)|\exp(p\alpha/n_0) \geq e^2$$  
        and, setting $j=p$, we are done, for appropriately chosen $N$. 

        Otherwise, $1 \leq k<p$. Since $f^k(z) \notin V$, 
        Lemma~\ref{lem:gammadfk} provides a constant $\beta > 0$ (for $\theta = \delta$, say) for which $|Df^k(z)| > \beta$. Moreover, $f^k(z) \in Z := (B(0,3) \cup f(V))\setminus V$. Therefore 
         it suffices to show that there is an $N$ such that, for each $y\in Z$, there is a $j \leq N$ with $|Df^j(y)| > 3/\beta$. 

         By Lemma~\ref{lem:Mdev2}, we can choose $N_1$ large enough that $Z \subset B(0,N_1)$ (trivially) and that, for any $z \in \ccc$, if $|f^n(z)| \geq N_1$ then $|Df^n(z)| > 3/\beta$.
         Thus we  restrict our attention to those $y$ which do not leave $B(0,N_1)$ for the first $N$ iterates, for some large $N$ to be defined.  
         We can cover the compact set $W:=\overline{B(0,N_1)}\setminus V$ by a finite collection of balls $\{W_l\}_{l=1}^L$ on which the first return map is expanding, by Lemma~\ref{fact:nice}, so there is a $\gamma >1$ and each return map $\phi_l : W_l \to W_l$ has derivative greater than $\gamma$.    

         Let $q,r \in \N$ satisfy $\beta \gamma^q >3/\beta$ and $\beta e^{r\alpha } |\lambda_0|e^{-N_1} >   3/\beta$. Set $N := qLrn_0$. 

        Consider the successive passages of $y$ into $W$, at times $k_0, k_1, \ldots, $ say. By time $k_{qL}$, if such exists, there must be some $W_l$ which is passed through at least $q$ times. 
        Then $|Df^{k_{qL}}(y)| > \beta \gamma^q > 3/\beta$ and if $k_{qL} \leq N$ we are done. 

        Otherwise, at some point the orbit must spend a long period, at least $rn_0$ long, in $B(0,N_1)\setminus W \subset V$. That is, there is some $a \geq 0$ 
        such that $f^l(y) \in V$ for $l = a+1, \ldots, a+ rn_0 < N$ and  
        such that $a = 0$ or $f^a(y) \in W$. Since $f^a(y) \in B(0,N_1)$,  $|f^{a+1}(y)| \geq |\lambda_0|e^{-N_1}$.  
        But by definition of $V$, 
        $$
        |Df^{rn_0}(f^{a+1}(y))| \geq \exp( r\alpha).$$
        The choice of $r$ entails $\beta |Df^{rn_0}(f^{a+1}(y))||f^{a+1}(y)| > 3/\beta $, so $|Df^{a +1 + rn_0}(y)| > 3/\beta$.  Noting that $a +1 + rn_0 \leq N$, we conclude the proof. 
    \end{proof}

        Recall that $\Delta >1$ is the constant giving a Koebe distortion bound of 2. 
        The following two lemmas are stated for maps in a neighbourhood of $f$ which are uniformly expanding on $B(P(f), 3\delta)$, see Section~\ref{sec:glob}. 

        \begin{lem} \label{lem:Vballs}
            Given $\varepsilon >0$ there is a $\delta_0 \in (0, \delta)$ such that the following holds. Let $\lambda \in B(\lambda_0, \eps_0)$. 
            If $f_\lambda^{k}(z) \in V$  for all $0 \leq k \leq p$ then there is a neighbourhood $W_p$ of $z$, contained in $B(z,\varepsilon/|Df_\lambda^p(z)|)$, mapped biholomorphically by $f_\lambda^{p}$ onto $B(f_\lambda^{p}(z), \Delta \delta_0)$. 
        \end{lem}
        \begin{proof} 
            First we consider a bounded number of iterates. 
            The distortion of $f_\lambda$ on $V$ is uniformly bounded (independently of $\lambda \in B(\lambda_0, \eps_0)$). Therefore, given $\varepsilon >0$ there is a $\delta_0 \in (0, \varepsilon/2\Delta)$ such that, if $y, f_\lambda(y), \ldots, f_\lambda^j(y) \in V$ and $0 \leq j \leq n_0-1$, there is a neighbourhood of $y$ contained in $B(y,2\delta)$ which is mapped biholomorphically by $f_\lambda^j$ onto $B(f_\lambda^j(y), \delta_0 \Delta^2)$. 
            
            Meanwhile, $|Df_\lambda^{n_0}| > 1$ on $B(V, 2\delta)$. It follows that, writing $p = an_0 +j$ with $a,j \in \N$ and $0 \leq j \leq n_0-1$, a neighbourhood of $z$ is mapped biholomorphically by $f_\lambda^{an_0}$ onto $B(f_\lambda^{an_0}(z), 2\delta)$. Combined with the previous paragraph, we deduce that a neighbourhood of $z$ is mapped by $f_\lambda^p$ biholomorphically onto $B(f_\lambda^p(z), \delta_0 \Delta^2)$. Shrinking the target, a neighbourhood $W_p$ of $z$ is mapped 
            by $f_\lambda^p$ biholomorphically with distortion bounded by $2$ onto $B(f_\lambda^p(z), \delta_0 \Delta)$. Because of the distortion bound, 
            $$W_p \subset B(z, 2 \delta_0 \Delta/ |Df_\lambda^p(z)) \subset B(z, \varepsilon/|Df_\lambda^p(z)|),$$ as required.
        \end{proof}
        \begin{lem}\label{lem:preVballs}
            There exists $\delta_0>0$ such that if $\lambda \in B(\lambda_0, \eps_0)$, if $f_\lambda^j(z) \in V$ for $j =1,\ldots, k$ and if $|Df_\lambda^k(z)| > 1$, 
            then there is a neighbourhood $U$ of $z$ mapped biholomorphically by $f_\lambda^k$ onto $B(f_\lambda^k(z), \Delta \delta_0)$ with $U \subset B(z, \delta)$. 
        \end{lem}
            \begin{proof}
                By hypothesis,  $ |Df_\lambda^k(z)| = |f_\lambda(z)||Df_\lambda^{k-1}(f_\lambda(z))| > 1$, so letting $\varepsilon < \delta/e$ and taking $\delta_0$ from the preceding lemma, there is a neighbourhood $W$ of $f_\lambda(z)$ contained in $B(f_\lambda(z), \varepsilon |f_\lambda(z)|)$ mapped biholomorphically onto $B(f_\lambda^k(z), \Delta \delta_0)$. 
                Since $f_\lambda$ is an exponential map, $f_\lambda(B(z,\delta)) \supset B(f_\lambda(z), |f_\lambda(z)|(1 - e^{-\delta}))$. Since $0< \delta <1$, $1-e^{-\delta} > \delta/e$. By choice of $\varepsilon$, we deduce that $B(f_\lambda(z), \varepsilon |f_\lambda(z)|) \subset f_\lambda(B(z, \delta))$, so $W \subset f_\lambda(B(z,\delta))$. 
                Therefore the relevant pullback $U$ of $W$ (that is, with $z \in U$)  is contained in $B(z, \delta)$. 
            \end{proof}

            The following lemma requires that the postsingular set is contained in $V$, so it only holds for $f = f_{\lambda_0}$. 
        \begin{lem} \label{lem:allballs}
            Let $\delta_0 >0$ be given by Lemma~\ref{lem:preVballs}.
            Let $z \in \ccc$ and suppose $|Df^k(z)| > |Df^j(z)|$ for all $j = 0,  \ldots k-1$. Then there is a neighbourhood of $z$ mapped biholomorphically by $f^k$ onto $B(f^k(z), \Delta \delta_0)$.  
        \end{lem}
        \begin{proof}
            If $f^j(z) \in V$ for $j = 1, \ldots, k$,  
            Lemma~\ref{lem:preVballs} produces the required neighbourhood. Otherwise, 
            there is a maximal $j \leq k$ for which $f^j(z)\notin V$. 
            By
            Lemma~\ref{lem:preVballs} again,
            there is a neighbourhood $U$ of $f^j(z)$ mapped by $f^{k-j}$ biholomorphically onto $B(f^k(z), \Delta \delta_0)$, and $U \subset B(f^j(z), \delta)$. But $f^j(z) \notin V$,  so $U \cap P(f) = \emptyset$. Therefore there is a neighbourhood of $z$ mapped by $f^j$ biholomorphically onto $U$. 
        \end{proof}

 \section{First entry to a right-half plane} \label{sec:FE}
        Proposition~\ref{lem:M0} is the principal result of this section. It states that a large proportion of points in a neighbourhood of $P(f)$ get mapped, in not too long time, far out to the right  and with derivative which is not too large. The idea behind the proof is porosity: at every small scale, a certain proportion gets mapped far out. We use the expansivity estimates from the previous section to transfer estimates from the  large scale to the small scale. We upgrade the proposition in Lemma~\ref{lem:xplus} both topologically, obtaining a well-behaved partition, and distance-wise, showing most points land a little further to the right than claimed by the proposition. 
	In the following two sections we will examine the dynamics far out to the right and obtain estimates for first entry maps to a (far) left half-plane.

        Let $\delta_0>0$ be the minimum of the $\delta_0$ given by Lemma~\ref{lem:allballs} and by Lemma~\ref{lem:Vballs} (with $\eps < 1/2$, say). 
        Let $N_1$ be given by Proposition~\ref{prop:hyp}.
        By Lemma~\ref{lem:Mdev2}, there is an $M>100$ such that, if $|f^n(z)|> M$ then 
        $ 
            |Df^n(z)| > |f^n(z)|/\delta_0. $ 
        We can suppose moreover that $$M>|\lambda_0|e^{N_1 + \diam(P(f)) + 10\Delta}.$$ This choice of $M$ is for future use [which the reader may choose to remember  as a \emph{sufficiently large constant}], for example to obtain \eqref{eq:Mbig} in the proof of Lemma~\ref{lem:Muse}.

        Recall $\cQ, \cR,$ and $\cL$ are globally defined in Section~\ref{sec:glob}.
            
            \begin{lem} \label{lem:Uballs}
                There is a finite collection of sets $U_1, \ldots, U_p$ with corresponding  numbers $n_k \geq 0$, $k = 1, \ldots, p$,  such that $f^{n_k}$ maps $U_k$ biholomorphically onto an element of $\cQ$ contained in $\cR(M)$, and such that for each $y$ with $|y| \leq 2M$, $B(y, \delta_0)$ contains some $U_k$. 
            \end{lem}
            \begin{proof}
            By transitivity of $f$, there is  finite set $Z$ such that $\dist(y, Z) < \delta_0/2$ for all $y$ with $|y| \leq 2M$, and such that for each $z \in Z$, there is an $n$ such that $f^n(z) \in \cR(M +2\pi)$. 
            For such $z,n$, let $Q \in \cQ$ be the square containing $f^n(z)$, so $Q \subset \cR(M)$. By choice of $M$ and by Lemma~\ref{lem:distnQ},  there is a neighbourhood $U$ of $z$ which gets mapped by $f^n$ biholomorphically onto $Q$ with distortion bounded by $2$. Since $|Df^n(z)| > |f^n(z)|/\delta_0 > M/\delta_0$, we deduce that the diameter of $U$ is bounded by $2\sqrt{2} \pi 2 \delta_0/M < \delta_0/2$.  Thus if $|y-z| < \delta_0/2$, $U \subset B(y,\delta_0)$. 
            The result follows.
        \end{proof}
        The following lemma deals with points in $\cL(M)$. We shall deal with points to the right subsequently. 
        \begin{lem} \label{lem:UKK}
            There is a countable collection of sets $\{U_i\}_{i \in \Z }$ and a constant $C>1$ such that the following holds. Each $U_i$ is mapped by some $f^n$, $n \geq 0$, onto a square $Q \in \cQ$ with $Q \subset \cR(M)$ with derivative bounded by $C$ and distortion bounded by $2$. If $\Re(y) \leq M$, then $B(y, \delta_0)$ contains as a subset an element of $\{U_i\}_{i\geq0}$.
        \end{lem}
        \begin{proof}
            Let $U_k, n_k$ for $k = 1, \ldots, p$ be given by Lemma~\ref{lem:Uballs}.
            Taking translates by multiples of $2\pi i$ of the sets $U_1, \ldots U_p$ deals with the points $y \in \ccc$ with $-M \leq \Re(y) \leq M$. 
        
            If $\Re(y) < -M$, by Proposition~\ref{prop:hyp} there is a least $j\geq 1$ with $3 < |Df^j(y)|$ and  for this $j$, $|f^j(y)|, |Df^j(y)| < N_1 +1 < M/2$.
            By Lemma~\ref{lem:allballs}, there is a neighbourhood of $y$ mapped biholomorphically by $f^j$ onto $B(f^j(y), \Delta \delta_0)$ with a corresponding sub-neighbourhood $W$ mapped by $f^j$ onto $B(f^j(y), \delta_0)$ with distortion bounded by 2 (by choice of $\Delta$). 
            On $W$ we deduce $1 < 3/2 < |Df^j| < M$. The lower bound implies $W \subset B(y, \delta_0)$.  Now $B(f^j(y), \delta_0)$ contains some $U_k$, with $1 \leq k \leq p$. Thus there is some $U_y \subset W$ mapped by $f^j$ onto $U_k$. The derivative $|Df^{j + n_k}|$ on $U_y$ is bounded by $2N_1 \sup_{U_k}|Df^{n_k}|$. Thus one can take $C := M \max_{1 \leq k \leq p} \sup_{U_k}|Df^{n_k}|$. 

            Countability of the collection of $U_y$ obtained follows from countability of $\cQ$ (and its preimages). The distortion bound comes from Lemma~\ref{lem:distnQ}.
        \end{proof}

        \begin{lem} \label{lem:squares0}
            Let $Z$ denote the cone of positive linear combinations of $1+i$ and $1-i$. 
            Let $y \geq M$. 
            Let $Q \in \cQ$ satisfy $Q \subset \cR(y) \setminus \cR(y+7)$. Then there is a subset of $Q$ mapped biholomorphically onto a square $Q' \in \cQ$ satisfying $Q' \subset Z \cap \cR(|\lambda_0|e^y /2) \setminus \cR(|\lambda_0|e^y e^7)$.
        \end{lem}
        \begin{proof}
            One quarter of any square of $\cQ$  gets mapped injectively into $Z$. We have $f(Q) \cap Z \subset \cR(|\lambda_0| e^y/\sqrt{2})$, and $f(Q) \cap \cR(|\lambda_0|e^{y+7}) = \emptyset$.  Only a small proportion of squares from $\cQ$ in $f(Q)\cap Z$ intersect $f(\partial Q)$, so we can pull back one of the other squares to get the required subset. 
        \end{proof}
        \begin{lem} \label{lem:squares}
            Suppose $Q \in \cQ$ satisfies $Q \subset \cR(M)$. Let $x > M$. For some $z \in Q$ and some $k\geq 0$, the ball $B(z, 1/x^3) \subset Q$ is mapped by $f^k$ univalently into $\cR(x)$. 
        \end{lem}
        \begin{proof}
            Suppose $Q \subset \cR(y) \setminus \cR(y+7)$. We can assume $y < x$, otherwise the statement holds trivially, with $k=0$.  By repeatedly applying Lemma~\ref{lem:squares0}, we can construct an increasing sequence of numbers $y=y_0 < y_1 < y_2 < \cdots$ and 
            a decreasing sequence of sets $Q = V_0 \supset V_1 \supset \cdots$ such that the following holds. 
            For each $k \geq 0$, 
            \begin{itemize}
                \item $f^k(V_k) \in \cQ$;
                \item
                      $f^k(V_k) \subset \cR(y_k)\setminus \cR(y_k +7)$;
                \item
                    $|\lambda_0| e^{y_k/2} < y_{k+1} < e^7|\lambda_0| e^{y_k}$;
                    \item
                        $|f^{k}(z)| \leq \sqrt{2} (y_k +7)$ for $z \in V_k$ (noting $f^k(z)$ is in the cone $Z$);
                  \item
                  the distortion of $f^k$ on $V_k$ is bounded by 2 (by Lemma~\ref{lem:distnQ}).
                  \end{itemize}
                  Since $\sqrt{y_{j+1}} > \sqrt{|\lambda_0|} e^{y_j/4} > 4 \sqrt{2}(y_j +7) > y_j$, we deduce that 
                  $$
                  \prod_{j=1}^k \sqrt{2} (y_j + 7) < ((y_k +7)/2) \prod_{j=1}^{k-1} \sqrt{y_{j+1}} \leq (y_k+7) y_k/2 < y_k^2.$$
	Thus on $V_k$ the derivative bound $$|Df^k| = \prod_{j=1}^k |f^j(z)| \leq \prod_{j=1}^k \sqrt{2}(y_j+7) < y_k^2$$
	applies. 

                  Let $k\geq 1$ be minimal such that $y_k \geq x$. If $f^k(V_{k-1}) \subset \cL(2e x)$ (equivalently, if $f^k(V_{k-1}) \cap \cR(2e x) = \emptyset)$ then $y_k \leq 2e x $ and  $|Df^k|$ on $V_k$ is bounded by $(2e x )^2$. Therefore $V_k$ easily contains a ball of radius $1/x^3$.
            Otherwise, $f^k(V_{k-1})$ is a geometric annulus centred on zero and intersecting $\cR(2e x)$, and the square $f^{k-1}(V_{k-1})$ contains a ball of radius $1/16$ mapped by $f$ into $\cR(x)$, as is easy to check. The derivative of $f^{k-1}$ on $V_{k-1}$ is bounded by $y_{k-1}^2 < x^2$, so pulling back the ball we get a set containing a ball of radius $1/x^3$ once again, as required. 
        \end{proof}
        \begin{lem} \label{lem:proportion1}
            There is a constant $\gamma > 0$ such that if $x > M$ the following holds.
            
            If $Q \in \cQ$, there is a ball of radius $\gamma / x^3$ inside $Q$ which gets mapped univalently by $f^n$, for some $n \geq 0$, into $\cR(x)$ with distortion bounded by 2. 

            If $\Re(y) < M$, then there is a ball of radius $\gamma/x^3$ inside $B(y, \delta_0)$ which gets mapped univalently by $f^n$, for some $n \geq 0$,  into $\cR(x)$ with distortion bounded by 2. 
        \end{lem}
        \begin{proof}
            This follows from Lemmas~\ref{lem:UKK}~and~\ref{lem:squares}.
        \end{proof}
        The preceding lemma says that a certain proportion of everything at the large scale gets mapped far out to the right. 
        The next lemma deduces the same, but at small scales.

    \begin{lem} \label{lem:Muse}
        There are constants $\kappa>0, M_0 \geq M$ such that the following holds. Given $r \in (0,1)$,  $x\geq M_0$ and $z \in \ccc$, there is a finite collection of pairwise-disjoint balls $B_i \subset B(z,r)$, each of radius $> e^{-2x}r$, and numbers $n_i \geq 0$ such that 
        \begin{itemize}
            \item
                $m(\bigcup_i B_i) / m(B(z,r)) > \kappa / x^6;$ 
            \item
                  $f^{n_i}$ maps $B_i$ univalently into  $\cR(x)$; 
            \item
                $|Df^{n_i}_{|B_i}| < e^{3x}/r$. 
        \end{itemize}
    \end{lem}
        \begin{proof}
	Note first that if $f^k$ maps a ball $B$ into $\cR(x)$, then $f^k$ is univalent on $B$, as $P(f) \cap \cR(x) = \emptyset$. 

                Let $n$ be minimal such that $|Df^n(z)| > 20 /r$. If there is some minimal $k< n$ with $f^k(z) \in \cR(x)$, we can just pull back $B(f^k(z), 1)$ to get a set containing $B(z, r/40)$, using the derivative estimate and a distortion bound of 2. Some large sector of $B(z, r/40)$ gets mapped by $f^k$ to $\cR(\Re(f^k(z)))$ and the 
                lemma follows easily. 
                
                Otherwise, 
                $f^{n-1}(z) \notin \cR(x)$, implying 
                \begin{equation} \label{eq:jkl}
                    |Df^n(z)| \leq |\lambda_0| e^{x}20/r,
                    \end{equation}
                    a bound we use later in the proof.

                If $|f^n(z)| < M$, then $f^n$ maps some neighbourhood $W$ of $z$ univalently onto $B(f^n(z), \delta_0)$ with distortion bounded by 2, by Lemma~\ref{lem:allballs}. 
                With $\gamma$ given by Lemma~\ref{lem:proportion1}, for some $j \geq 0$ there is a ball of radius $\gamma/x^3$ in $B(f^n(z), \delta_0)$ which gets mapped by $f^j$ with distortion bounded by 2 into $\cR(x)$. As $|Df^n| < 2|f^n(z)| 20/r < 40M/r$ on $W$,  pulling  back this ball gives a subset of $W$ containing a ball of radius $(\gamma/x^3) r/40M$, as required. 

                Now we treat the case $|f^n(z)| \geq M$. 
                Let $r' \leq r$ be maximal such that $f^{n-1}(B(z,r')) \subset B(f^{n-1}(z), 1)$.  Set $W := B(z,r')$. As a neighbourhood of $z$ gets mapped biholomorphically onto by $f^{n-1}$ onto $B(f^{n-1}(z), 1)$ and $f$ is univalent on each ball of radius $1$,  $f^n$ is biholomorphic on $W$.
                Since 
		\begin{equation} \label{eq:Mbig}
		|f^{n-1}(z)| \geq \re(f^{n-1}(z)) > \diam(P(f)) + 10\Delta
		\end{equation}
		by choice of $M$, Lemma~\ref{lem:distnQ} implies that the distortion of $f^{n-1}$ on $W$ is bounded by $2$. Thus $|Df^{n-1}| < 40/r$ on $W$, so  $W \supset B(z,r/40)$. The distortion of $f$ on any ball of radius $1$ is $e^2$, so the distortion of $f^n$ on $W$ is bounded by $2e^2$.
               
                The advantage of choosing $W$ in this way is due to the distortion bound: 
                if we can show $f^n(W)$ contains at least one square $Q \in \cQ$, then 
                 the squares $$\{Q \in \cQ : Q \subset f^n(W)\}$$ fill some definite proportion of $f^n(W)$. 
                We now have two further subcases. 
                
                Suppose first that $r' = r$, so $W =  B(z,r)$. 
                There is a  $Q \in \cQ$ containing $f^n(z)$, so (by Lemma~\ref{lem:distnQ}, as usual) a neighbourhood $W_z$ of $z$ gets mapped biholomorphically onto $Q$ by $f^n$ with distortion bounded by 2. Since $|Df^n(z)|> 20/r$, we deduce that $\diam(W_z) < r \diam(Q)/10$, hence $W_z \subset B(z,r) = W$. In particular, $f^n(W)$ contains at least one square from $\cQ$. 

                If we assume, on the other hand, that $r' < r$, then $f^{n-1}(W) \supset B(f^{n-1}(z), 1/2)$, by bounded distortion, and $f^n(W)$ is huge, in particular it contains at least one square $Q \in \cQ$. 

                 We have shown that in both subcases (so whenever $|f^n(z)| \geq M$),
                 the squares $\{Q \in \cQ : Q \subset f^n(W)\}$ fill some definite proportion of $f^n(W)$. 
                Consequently, there is some independent constant $\gamma' >0$ and  a collection of pairwise-disjoint subsets $W_i \subset  W$,  each mapped by $f^n$ onto an element $Q_i$ of $\cQ$ with $m(\bigcup_i W_i)/m(W) > \gamma'$, say. 
                One can apply Lemma~\ref{lem:proportion1} on each $Q_i$ 
		to obtain a ball $B(y, \gamma/x^3)\subset Q_i$ say and some $l\geq 0$ such that $f^{l}$ maps the ball univalently  into $\cR(x)$. Let $Z_i := B(y, \gamma/\Delta x^3)$ and let 
		 $V_i = W_i \cap f^{-n}(Z_i)$. By the Koebe principle, if $j,k \geq 0$ and $j+k \leq n+l$, the distortion of $f^j$ is bounded by $2$ on $f^k(V_i)$.

                The distortion bound  implies $V_i$ contains a ball $B_i$ of radius $\diam(V_i)/4$, so $m(B_i)/m(V_i) > 1/16$. 
                The bound (\ref{eq:jkl}) gives a bound on $|Df^n_{|B_i}|$ of $|\lambda_0|e^{x} 40/r$, which implies $B_i$ has radius $ \geq (\gamma/\Delta x^3) r/40|\lambda_0| e^{x} > e^{-2x} r$, provided $x$ is large enough. 
		This is the required estimate on the radii. 

		Continuing on, let $k \leq n+l$ be minimal such that $f^k(B_i) \subset \cR(x)$. 
		Thus there is a point in $f^{k-1}(B_i)$ not in $\cR(x)$, so, by bounded distortion, $|Df| < 2|\lambda_0| e^x$ on $f^{k-1}(B_i)$. 
		Univalence on $f^{k-1}(B_i)$ implies this set does not contain a ball of radius $\pi$, so the distortion bound of $2$ for $f^{k-1}$ on $B_i$ and the estimate for the radius of $B_i$ combine to imply
                $$
                |Df^{k-1}_{|B_i}| < 80\pi |\lambda_0| e^{x}x^3 \Delta /r\gamma. 
                $$
                Thus $|Df^k_{|B_i}| < 160\pi |\lambda_0|^2 e^{2x} \Delta x^3/r\gamma < e^{3x} /r$, if $x$ is large enough.  

                We note to finish that $m(Q_i) = 4\pi^2$ while $m(f^n(V_i)) = \pi \gamma^2/\Delta^2 x^6$, so $$m(V_i)/m(W_i) > \gamma^2/\Delta^2 x^6 16\pi$$ for each $i$. Combining this with the uniform estimates for $m(B_i)/m(V_i)$, $m(\bigcup_i W_i)/m(W)$ and $m(W)/m(B(z,r)$, we conclude  $m(\bigcup_i B_i) /m(B(z,r))  > \kappa/x^6 $ for some $\kappa>0$ independent of $x$. This completes the proof of the case $|f^n(z)| \geq M$. 
        \end{proof}
        We call a square $D$ \emph{dyadic} if $2\pi 2^k D$ is an element of $\cQ$ for some integer $k \geq 1$; $2^{-k}$ is then called the \emph{scale} of $D$. 
        Since each ball contains a square of comparable size, and vice versa, the previous lemma also holds for dyadic squares, with perhaps a slightly smaller scale (which we estimate crudely). 

        \begin{lem}\label{lem:dyadic}
            There are constants $ \kappa>0, M_0 \geq M$ such that the following holds. Let $k\geq 3$. Let $x\geq M_0$ and let $D$ be a dyadic square of scale $2^{-k}$.  Then there is a finite collection of pairwise-disjoint dyadic squares $D_i \subset D$, each of scale $>  e^{-3x}2^{-k}$, such that 
        \begin{itemize}
            \item
                $m(\bigcup_i D_i) / m(D) > \kappa / x^6;$ 
            \item
                for each $D_i$ there is an $n_i\geq0$ with $f^{n_i}(D_i) \subset \cR(x)$; 
            \item
                $f^{n_i}$ is univalent on $B(z, \Delta \diam(D_i))$ for all $z \in D_i$;
            \item
                $|Df^{n_i}_{|D_i}| < e^{3x}2^k$. 
                \end{itemize}
        \end{lem}

        If at all scales, a certain proportion gets mapped far out to the right, then almost every point does. The next lemma gives bounds on the time needed for a \emph{large} proportion of points to get mapped far out to the right, together with a bound on the corresponding derivatives.

        \begin{prop} \label{lem:M0}
            Let $S$ be a bounded set. There is a  constant $M_0$ such that the following holds. 
            Let $x > M_0$. 
            Let $S_*$ denote the set of points $z$ such that 
                    the first entry to $\cR(x)$ happens at time $n(z)$ with 
            \begin{itemize}
                \item
                    $|Df^{n(z)}(z)| < e^{x^9}/2$;
                \item
                    $n(z) \leq e^{2x}$. 
                \end{itemize}
            Then $m(S \setminus S_*) \leq  1/x$. 
        \end{prop}
        \begin{proof}
            Let $\kappa, M_0$ come from Lemma~\ref{lem:dyadic}.
            We can cover $S$ with a finite number of dyadic squares of scale $2^{-3}$, each contained in $B(S,1)$, and with total area $a$, say. 
            If $M' > M_0$ is sufficiently large, $x > M'$ and $p = x^{7}$, then 
            $$
            (1 - \kappa/x^6)^p a < e^{-\kappa x/2}a< 1/x.
            $$
            At least a proportion $\kappa/x^6$ of each of these dyadic squares is covered by dyadic squares of scale $\geq 2^{-3} e^{-3x}$ given by   
            Lemma~\ref{lem:dyadic}. The remainder, less than $(1-\kappa/x^6)$, can be covered by other dyadic squares of scale $\geq 2^{-3}  e^{-3x}$ and we can apply Lemma~\ref{lem:dyadic} to each of these squares. 
            Proceeding inductively, after $p$ such applications, we end up with a collection $\cD$ of dyadic squares such that 
            $$
            m\left(S \setminus \bigcup_{D\in \cD}D\right) \leq (1 - \kappa/x^6)^p a  < 1/x
            $$
            and such that
            each $D \in \cD$ satisfies
            \begin{itemize}
                \item
                    the scale of $D$ is $\geq 2^{-3}  (e^{-3x})^p$;
                \item
                    there is an $n_D\geq0$, with $f^{n_D}(D) \subset \cR(x)$;
                \item
                $f^{n_D}$ is univalent on $B(z, \Delta \diam(D))$ for all $z \in D$;
            \item
                $|Df^{n_D}_{|D}| < (e^{3x})^{p+1}$. 
                \end{itemize}
                We wish to show that $S_*$ contains $\bigcup_{D \in \cD} D$. 
                For a point $y \in D \in \cD$, $n_D$ is not necessarily the first entry time $n(y)$ to $\cR(x)$, but for all $j < n_D$, Lemma~\ref{lem:Mdev} implies  $3|Df^j(y)| < |Df^{n_D}(y)|$, so $|Df^{n(y)}(y)| < (e^{3x})^{p+1} < e^{x^9}/2$.

                It remains to show that $n_D$ is not too large. 
                It can be assumed that $n_D$ is minimal such that $f^{n_D}(D) \subset \cR(x)$. Now $f^j$ on $B(z,\diam(D))$ is univalent with distortion bounded by 2 
                for all $j\leq n_D$, by choice of $\Delta$, so $f^j(B(z,\diam(D)))$ cannot contain a ball of radius $\pi$ for any $j < n_D$ and thus has diameter bounded by $4\pi$. In particular, it does not intersect $\cR(x+4\pi)$. 
                Thus for $1 \leq j \leq n_D$, $f^j(D) \subset B(0, |\lambda_0|e^{x+4\pi})$.

                By Proposition~\ref{prop:hyp}, inside the region $B(0, |\lambda_0|e^{x+4\pi})$ the derivative multiplies by at least 3 at least every $C_0 e^{x}$ steps for some $C_0>0$.   
                Therefore
            $$
            3^{n_D/C_0e^{x}} < |Df^{n_D}_{|D}|, 
            $$
            so taking logs and using the estimate for the derivative,
            $$
            n_D/C_0e^{x} < 3x(p+1), 
            $$
            $$
            n_D < C_0  (p+1)(3x) e^{x} < e^{2x}, 
            $$
	    provided $x$ is large enough, $x > M''$ say.
            We reset $M_0 := \max(M',M'')$. 
        \end{proof}

        Next we show that the first entry usually happens a bit further to the right, and we recover some Markov property (\emph{equal or disjoint}) which keeps the subsequent arguments from getting too messy. 
        \begin{lem} \label{lem:xplus}
            Given $C  > 0$, there exists $M_0$ such that, if $A \subset B(P(f), 1)$ is a simply-connected open set with $\partial A$ of length at most $C$, then for all $x > M_0$ the following holds. 

            There exists a set $A_* \subset A \setminus B(\partial A, x^{-1/4})$ and a partition $\cW$ of $A_*$ into elements $W$ with associated numbers $n_W$, such that
            \begin{itemize}
                \item 
                    $m(A\setminus A_*) < 1/2 \log x$;
                \item
                    $|Df^{n_W}| < e^{x^9}$ on $W$;
                \item
                    $n_W \leq e^{2x}$;
                \item
                    $n_W$ is the first entry time to $\cR(x+\log \frac32)$;
                \item
                    $f^{n_W}$ maps $W$ biholomorphically onto a square from $\cQ$;
                \item
                    $f^{n_W}(W) \subset \cR(x + 2\sqrt{x})$.
            \end{itemize}
        \end{lem}
        \begin{proof}
            In the proof, the sets $W$ obtained will be mapped biholomorphically by corresponding $f^{n_W}$ onto unions of squares of $\cQ$ rather than onto single squares. This is of no import, as there will be a subpartition of each $W$ whose elements each get mapped by $f^{n_W}$ onto an element of $\cQ$. 

            For large $x$, a standard estimate for the area of a tubular neighbourhood gives
            $$m(B(\partial A, 2x^{-1/4})) \leq 4 C x^{-1/4} + 4 \pi x^{-1/2} < 8C x^{- 1/4}.$$
            Therefore, setting
            $A_x:=A \setminus B(\partial A, 2x^{-1/4})$, we have $m(A_x) > m(A) - 8Cx^{-1/4}$. 

            Let $S_*$ be given by Proposition~\ref{lem:M0} for $S = B(P(f),1)$ (and $x$ sufficiently large). 
            Set $A' : = S_* \cap A_x$, so 
	    \begin{equation}\label{eq:AA14}
	    m(A \setminus A') <  1/x + 8Cx^{-1/4}.
	    \end{equation}
	    Let $z \in A'$ and let $n_0 = n_0(z)$ be the associated number $n(z)$ given by Proposition~\ref{lem:M0}. Then $n_0$ is the first entry time of $z$ to $\cR(x)$, while $n_0 \leq e^{2x}$  and 
	    \begin{equation}\label{eq:dbn0}
	    |Df^{n_0}(z)| < e^{x^9}/2. 
	    \end{equation}


            Suppose first, in case one, that $\Re(f^{n_0}(z)) < 2\pi \lfloor x + x^{3/4} /2\pi \rfloor$. Let $T$ denote the partial strip 
            $$
            \{w : x \leq \Re(w) <  2\pi \lfloor x + x^{3/4}/2\pi \rfloor; 2j\pi \leq \Im(w) < (2j+2)\pi\}
            $$
            containing $f^{n_0}(z)$, for the relevant integer $j$. By Lemma~\ref{lem:Mdev2}, the neighbourhood  $W_*$ of $z$ mapped univalently by $f^{n_0}$ onto $T$ has diameter less than  $x^{3/4}/x = x^{-1/4}$, while $z \in A_x$, so $W_* \subset A \setminus B(\partial A, x^{-1/4})$. 
            Let 
            $T_+ := T \cap \cR(2\pi \lfloor x + 3\sqrt{x} /2\pi \rfloor)$, and set $W_z := W_* \cap f^{-n_0}(T_+)$. 
            Note $W_z$ does not necessarily contain $z$. 
            Then $m(T \setminus T_+)/m(T) < 4\sqrt{x}/x^{3/4}$, so  
            $$m(W_z)/m(W_*) \geq 1 - 16x^{-1/4},$$ using a distortion bound of 2 from Lemma~\ref{lem:distnQ}.

            If, in case two, $\Re(f^{n_0}(z)) \geq 2\pi \lfloor x + x^{3/4} /2\pi \rfloor$,
            let $T$ denote the partial strip 
            $$
         \{w : 2\pi k \leq \Re(f^{n_0}(w)) < 2\pi (k+1); 2j\pi \leq \Im(w) < (2j+2)\pi\}
            $$
            containing $f^{n_0}(z)$, for the relevant integers $k, j$. As before, by Lemma~\ref{lem:Mdev2}, the neighbourhood  $W_z = W_*$ of $z$ mapped univalently by $f^{n_0}$ onto $T$ has diameter less than $1/x < x^{-1/4}$, 
            and $f^{n_0}$ on $W_z$ has distortion bounded by 2.
            Again we deduce $W_z \subset A \setminus B(\partial A, x^{-1/4})$. 

            In both cases, for $j < n_0$, $f^j(W_*)$ has diameter bounded by $2x^{3/4}/x < \log \frac32$, so $f^j(W_*) \cap \cR(x +\log \frac32) = \emptyset$. Meanwhile, $f^{n_0}(W_z) \subset \cR(x + 2\sqrt{x})$. In particular, on $W_z$, $n_0$ is the first entry time to $\cR(x + \log \frac32)$.

            We claim that for $z_1,z_2 \in A'$, the sets $W_1 = W_{z_1}, W_2 = W_{z_2}$ are either equal or  disjoint. Let $n_1 = n_0(z_1)$, $n_2 = n_0(z_2)$. 
            The partial strips $f^{n_1}(W_1), f^{n_2}(W_2)$ are either equal or disjoint. If $n_1 = n_2$ it follows that $W_1, W_2$ are either equal or disjoint. So suppose $n_1 < n_2$ and $W_1 \cap W_2 \ne \emptyset$. 
            But $f^{n_1}(W_1) \subset \cR(x +\sqrt{x})$, so
            $f^{n_1}(W_2) \cap \cR(x+\sqrt{x}) \ne \emptyset$, contradicting $f^j(W_2) \cap \cR(x +\frac32 ) = \emptyset$ for $j < n_2$. We conclude that  the claim holds.

            We thus obtain a (necessarily finite) pairwise-disjoint collection $\cW$ of (such) subsets $W \subset A \setminus B(\partial A, x^{-1/4})$ with 
	    \begin{equation}\label{eq:WA14}
	    m\left(\bigcup_{W \in \cW} W\right) = \sum_\cW m(W)  \geq (1 - 16x^{-1/4}) m(A') .
	    \end{equation}
            Set $A_* := \bigcup_{W\in \cW} W$. Together with \eqref{eq:AA14}, \eqref{eq:WA14} implies
	    \begin{equation*}
                \begin{split}
            m(A \setminus A_*) &\leq m(A) - m(A') + 16x^{-1/4} m(A') \\
             & <  1/x + 8Cx^{-1/4} + 16x^{-1/4}m(B(P(f),1)) \\
             & < 1/2\log x.
         \end{split}
	    \end{equation*}
	    
            If $W = W_z$ for some $z \in A'$, set $n_W := n_0(z)$, so $n_0 < e^{2x}$. 
            The distortion bound of $2$ combined with \eqref{eq:dbn0} gives the required derivative estimate $|Df^{n_W}| < e^{x^9}$ on $W$. 
        \end{proof}

\section{Far-right dynamics} \label{sec:FR}
    The  dynamics far to the right is relatively easy to understand (and long-known, see for example \cite{Rees86, Lyubich:Exp, McMullen:Area}). Far-right squares from $\cQ$ get mapped to enormous annuli, with approximately half getting mapped  to the far-far-left, and half getting mapped  to the far-far-right. That which gets mapped to the right, 
    subsequently half of it gets mapped farther to the left, half farther to the right, and so on. Thus most points far to the right get mapped reasonably quickly far, far to the left. A mathematical formulation is given by the following two lemmas. 
\begin{lem} \label{lem:annulus}
    Suppose $n, S$ are such that $f^n$ maps $S$ biholomorphically onto some $Q \in \cQ$.  Provided the real rumber $y$ satisfying $\re(Q) = [y, y+ 2\pi)$ is large enough, 
 there is a finite partition of $S$ into subsets $S_*, S_L, S_1, S_2,\ldots, S_p$ such that the following holds:
    \begin{itemize}
        \item
            $m(S_*) < m(S) / 2y$;
        \item
            $f^{n+1}(S_L) \subset \cL(-e^{y - \sqrt{y/2}})$;
        \item
            $m(S)/9 < m(S_1 \cup \cdots \cup S_p) < \frac78 m(S)$;
        \item
            each $S_l$, $1 \leq l \leq p$ is mapped by $f^{n+1}$ biholomorphically onto an element of $\cQ$ contained in $\cR(e^{y - \sqrt{y/2}})$;
        \item
            $|\re(f^{n+1}(z))|^2 > |f^{n+1}(z)|$ for all $z \in S \setminus S_*$. 
    \end{itemize}
\end{lem}
\begin{proof}
    The proof will use that $A := f(Q)$ is a gigantic annulus, so most of it (by area) is a long way from the imaginary axis. Note that on $Q$, the distortion of $f$ is bounded by $e^{2\pi}$, so on $S$, the distortion of $f^{n+1}$ is bounded by $2 e^{2\pi}$. 

    For  $r = |\lambda_0|e^y$, the annulus $A$ has inner radius $r$ and outer radius $r e^{2\pi}$. Its area is $\pi r^2 (e^{4\pi} -1)$. 
    Let $X$ be the subset of $A$ consisting of points close to the imaginary axis and close to $f(\partial Q)$ defined by
    $$X := \{z \in A: |\re(z)| \leq |\lambda_0|^{-1} r e^{-\sqrt{y/2}} + 2\pi \} \cup B(f(\partial Q), 2\pi).$$
    Then $m(X)$ is bounded by $2|\lambda_0| e^{2\pi} r^2 e^{-\sqrt{y/2}}$. 
    Thus $m(X)/m(A) < e^{-\sqrt{y/3}}$, say, for large $y$. From this and the distortion bound we deduce that $m(S \cap f^{-n-1}(X)) < m(S) /2y$, provided $y$ is large enough. 

    Set $S_L := f^{-n-1}(A \cap \cL(0) \setminus X)$. Then $f^{n+1}(S_L) \subset \cL(-e^{y - \sqrt{y/2}})$. 

    Let $Y$ be the union of squares from $\cQ$ containing points of 
    $A \cap \cR(0) \setminus X$.
    From the definition of $X$,  
    $ Y \subset A \setminus f(\partial Q)$ and  $Y \subset \cR(e^{y-\sqrt{y/2}})$. As 
    $$\frac49 m(Q) < m(f^{-1}(Y)\cap Q) < m(Q)/2,$$ using a distortion bound of $2$ we deduce  $m(S)/9 < m(f^{-n-1}(Y) \cap S) < \frac{7}8 m(S)$ (one could improve this estimate to approximately $\frac12m(S)$, but it is unnecessary). One can clearly partition the pullback of $Y$ into the required sets $S_1, \ldots, S_p$. 

    Set $S_* := S \setminus (S_L \cup S_1 \cup \cdots \cup S_p)$. Since $f^{n+1}(S_*) \subset X$, we have from above that $m(S_*) < m(S)/2y$. 

    For $z \in S \setminus S_*$, we have 
    $$ e^{y - \sqrt{y/2}} \leq |\re(f^{n+1}(z))| \leq |f^{n+1}(z)| \leq |\lambda_0| e^y e^{2\pi} < e^{3y/2} \leq |\re(f^{n+1}(z))|^2.$$ 
\end{proof}

The square root terms in the following lemma are not exactly elegant, but they are used in the proof of Proposition~\ref{lem:s0}. 
\begin{lem} \label{lem:QRtoL}
        Let $E : y \to e^{y}$. 
    Let $Q \in \cQ$ and suppose $Q \subset \cR(x+2\sqrt{x})$. 
    If $x >0$ is sufficiently large,  there is a set $Q_0 \subset Q$ such that $m(Q_0)/m(Q) > 1/x$ 
    and for all $z \in Q_0$, there is an integer $k = k(z)$ such that the following holds:
    \begin{itemize}
        \item
            $1 \leq k \leq x$;
        \item
            $f^k(z) \in \cL(-e^{x + \sqrt{x}}) \cap \cL(-E^k(x))$;
        \item
            $|Df^k(z)| < |f^k(z)|^2 < |\re(f^k(z))|^4$;
        \item
            $m(\{z \in Q_0 : k(z) \geq 4\}) > m(Q)/1000$. 
    \end{itemize}
    Moreover, for $1 \leq j < k$, 
        $f^j(z) \in \cR(E^j(x))$ and
            $|Df^j(z)| < |f^j(z)|^2$.
\end{lem}
    \begin{proof}
        Note that if $y \geq x + 2\sqrt{x}$, then 
	$$y - \sqrt{y/2} \geq x + 2\sqrt{x} - \sqrt{(x+ 2\sqrt{x})/2} > x + \sqrt{x}.$$
	Moreover $e^{y-\sqrt{y/2}} >  e^{x + \sqrt{x}} > e^x + 2\sqrt{e^x}$.  
        Inductively applying Lemma~\ref{lem:annulus}, we obtain 
        sets $Q = Y^0 \supset Y^1 \supset \cdots$ and a collection of
        pairwise-disjoint sets $S^0_L, S^1_L, \ldots, S_*^0, S_*^1, \ldots$ 
        for which 
    \begin{itemize}
        \item
            for $0 \leq j \leq l$,  $f^j(Y^l) \subset \cR\left(E^j(x) +2\sqrt{E^j(x)}\right) \subset \cR(E^j(x))$; 
        \item 
            $Y^l$ can be partitioned into sets mapped biholomorphically by $f^{l}$ onto squares from $\cQ$ (which together with the previous point allows one to proceed inductively); 
        \item
            $Y^l = S_L^{l} \cup S_*^{l} \cup Y^{l+1}$; 
        \item 
            $m(S_*^l) < m(Q) \left( \frac1{2E^{l}(x)}\right);$
        \item 
            $ m(Q)/9^{l} < m(Y_l) < m(Q) (\frac78)^{l}; $
        \item
            for $z \in S_L^l$, $f^{l+1}(z) \in \cL(-e^{x+ \sqrt{x}}) \cap \cL(-E^{l+1}(x))$;
        \item 
            for $z \in S_L^l$ and $1 \leq j \leq l+1$,  
            $
            |f^j(z)|
            < 
            |\re(f^j(z))|^2 $.
    \end{itemize}
        Thus
             $Y^l = 
             Q\setminus (S_L^0 \cup \cdots \cup S_L^{l-1} \cup S_*^0 \cup \cdots \cup S_*^{l-1})$.
    Set $Q_0 :=  S_L^0 \cup \cdots \cup S_L^{\lfloor x \rfloor -1},$
    so $Q_0 = Q \setminus \left( Y^{\lfloor x \rfloor} \cup S_*^0 \cup \cdots \cup S_*^{\lfloor x \rfloor} \right)$. 
    From the two measure estimates, 
    $$m( Q_0)/m(Q) > \left(1 - \left(\frac78\right)^{\lfloor x \rfloor} - \sum_{l \geq 0} \frac{1}{2E^l(x)} \right) > 1/x.$$
    For $z \in S_L^l$, we set $k(z) := l+1$. If $z \in Y^3 \setminus Q_*$ then $k(z) \geq 4$, and $m(Y^3 \setminus Q_*) /m(Q) \geq 9^{-3} - 1/x > 1/1000$. 

    It only remains to check the derivative. We have, for $z \in S_L^l$ and $1 \leq j \leq l$, 
    $$|f^{j}(z)|^2 < |\re(f^j(z))|^2 < |\lambda_0|e^{\re(f^j(z))} =  |f^{j+1}(z)|$$
    so, for $0 \leq j \leq l$,
    \begin{equation*}
    \begin{split}
        |Df^{j+1}(z)| & = \prod_{a=1}^{j+1} |f^a(z)| 
	 \\
	 & \leq |f^{j+1}(z)|^{ 1+ \frac12 + \frac14 + \cdots + 2^{-j}} 
         \\
         & \leq |f^{j+1}(z)|^2 \leq |\re(f^{j+1}(z))|^4,
     \end{split} 
     \end{equation*}
         as required.
 \end{proof}

 \section{First entry to the left half-plane} \label{sec:lhp}

 A key claim in the following proposition is that for many points, the first entry to $\cL(-|\lambda_0|e^x)$ actually lands in $\cL(-e^{x + \sqrt{x}})$. This added distance will be needed, see Lemma~\ref{lem:guide}.

        \begin{prop} \label{lem:s0}
            Given $C  > 0$, there exists $M_0$ such that, if $A \subset B(P(f), 1)$ is a simply-connected open set with $\partial A$ of length at most $C$, then for all $x > M_0$ the following holds. 
        There is a set $A_0$ of points $z \in A \setminus B(\partial A, x^{-1/4})$ such that 
                    the first entry to $\cL(-2|\lambda_0|e^x)$ happens at time $n(z)$ with 
            \begin{enumerate}
                \item
                    \label{enum:s0a}
                    $f^{n(z)}(z) \in \cL(-e^{x + \sqrt{x}})$
                \item
                    \label{enum:s0b}
                    $e^x < |Df^{n(z)}(z)| <  e^{x^9}|\Re(f^{n(z)}(z))|^4$;
                \item
                    \label{enum:s0c}
                    $n(z) \leq e^{3x}$;
		    \item
                    \label{enum:s0d}
		    there exists $n_0(z) < n(z)$ for which $|Df^{l}(z)| <  e^{x^9}$ for $l \leq n_0$ and for which,
		    for $l = n_0(z)+1, \ldots, n(z)$, 
		    $$ |Df^{l}(z)| <  e^{x^9} |f^l(z)|^2;$$
                \item
                    \label{enum:s0e}
                    $\inf_{j+k\leq n(z)} |Df^j(f^k(z))| > 2\exp(-2|\lambda_0|e^{x})$;
                \end{enumerate}
                and with $m(A \setminus A_0) \leq  1/\log{x}$. 
%
        \end{prop}
        \begin{proof}
            Let $A_*$,
            with its attendant partition $\cW$,
            be given by Lemma~\ref{lem:xplus}. 
             Let $W \in \cW$ and let $n_W$ be given by Lemma~\ref{lem:xplus}.
            Let $Q = f^{n_W}(W) \in \cQ$, and note $Q \subset \cR(x+2\sqrt{x})$.  Let $Q_0(W) = Q_0$ be given by Lemma~\ref{lem:QRtoL}.
            Set $$A_0 := \bigcup_{W \in \cW} W \cap f^{-n_W}(Q_0(W)).$$
            Then \ref{enum:s0a}-\ref{enum:s0d} are immediately obtained combining the estimates of Lemma~\ref{lem:xplus} and Lemma~\ref{lem:QRtoL}, with $n_0(z) = n_W$ for $z \in W$. 

            It remains to justify \ref{enum:s0e} and the measure estimate. Now $n_0(z)$ is the first entry time to $\cR(x+\log \frac32)$, so for $1 \leq j \leq n(z)$, 
            $$|f^j(z) | \geq |\lambda_0| \exp(- \frac32 |\lambda_0| e^x) > 
            2\exp(-2 |\lambda_0| e^x)/\beta_1 ,$$
            where $\beta_1$ comes from Lemma~\ref{lem:mindev}, and \eqref{eq:mindeveq} implies \ref{enum:s0e}.
            For the measure estimate, 
            note 
            $m(Q_0)/m(Q) > 1 - 1/x$ so, 
            with a distortion bound of 2 for $f^{n_W}$ on $W$, 
            $$
            m(A_0)/m(A_*) > 1 - 4/x.$$
            Meanwhile, $m(A \setminus A_*) <  1/2 \log x$ 
            and $m(A_*) < m(B(P(f),1))$ 
            so $$m(A \setminus A_0) < 1/2 \log x + m(A_*) 4/x < 1/\log x,$$
            as required.
        \end{proof}

 \section{Lyapunov exponents almost never exist} \label{sec:Lyap}
    In this section we prove Theorems~\ref{thm:LE1} and~\ref{thm:LE2}. 
        We shall use the fact that Lebesgue measure is conservative and ergodic, see \cite{MeBartek}, to go from statements about positive-measure subsets to statements about full-measure subsets.
    \begin{lem} \label{lem:LEzero}
        For almost every $z$ and any Riemannian metric $\rho$, 
        $$
    \limsup_{n \to \infty} \frac{1}{n} \log |D_\rho f^n(z)| \geq 0.$$
\end{lem}
    \begin{proof}
	By Lemma~\ref{lem:Mdev} say, there is an $M$ such that the first return map $\phi$ to $B(M+1, 1)$ has $|D\phi| > 3$. Since Lebesgue measure is conservative and ergodic, almost every $z$ enters $B(M+1, 1)$ infinitely often. Thus for amost every $z$, there is a sequence $n_k$ with $f^{n_k}(z) \in B(M+1,1)$ and $|Df^{n_k}(z)| \to +\infty$. Since $B(M+1,1)$ is bounded, $|Df_\rho^{n_k}(z)| \to +\infty$. 
    \end{proof}
    \begin{lem} \label{lem:LEinfty}
        For almost every $z$ and any Riemannian metric $\rho$, 
        \begin{equation}\label{eq:LEl1}
    \liminf_{n \to \infty} \frac{1}{n} \log |D_\rho f^n(z)| = -\infty.
    \end{equation}
    For almost every $z$ and the Euclidean metric,
        \begin{equation}\label{eq:LEl2}
    \limsup_{n \to \infty} \frac{1}{n} \log |D f^n(z)| = +\infty.
            \end{equation}
\end{lem}
    \begin{proof}
        Let $x > 0$ be large. 
        Let $A = B(0,1)$, say, and let $A_*$ and its attendant partition $\cW$ be given by  
        Lemma~\ref{lem:xplus}. Then $m(A_* ) > \pi/2$ say. Let $W \in \cW$ and let $n_W \leq e^{2x}$ be given by Lemma~\ref{lem:xplus}. Then 
        $|Df^{n_W}| < e^{x^9}$ on $W$, and $Q_W := f^{n_W}(W) \in \cQ$ and $Q_W \subset \cR(x + 2\sqrt{x})$. 

        By Lemma~\ref{lem:QRtoL} there is a subset $S_W \subset Q_W$ with $m(S_W) \geq m(Q_W)/1000$ for which the following holds. Let $z \in W \cap f^{-n_W}(S_W)$ and set $w: = f^{n_W}(z)$.  There is a $k = k(z)$ with $4 \leq k \leq x$,   
        \begin{itemize}
            \item
        $f^{k}(w) \in \cL(-E^4(x))$, where $E : y \mapsto e^{y}$;
    \item
        for $1 \leq j \leq k$,
             $|Df^{j}(w)| < |f^j(w)|^2 < |\re(f^j(w)|^4$.
        \end{itemize}
        Then (by Lemma~\ref{lem:Mdev2})
        $$|Df^{n_W + k}(z)| > |f^{n_W + k}(z)| > E^4(x).$$
    Meanwhile, $n_W + k \leq e^{2x} + x < 2 e^{2x}$. Thus 
        $$
        \frac1{n_W +k} \log |Df^{n_W + k}(z)| > E^3(x)/2e^{2x} \gg x.$$
        Going one step further will give us a tiny derivative. 
    \begin{equation*}
        \begin{split}
        |Df^{n_W + k+ 1}(z)| & \leq e^{x^9} |\re(f^k(w)|^4
        |\lambda_0| \exp (\re(f^k(w)))
        \\ &
        \leq e^{x^9} \exp (\re(f^k(w)) /2) 
        \\ &
        \leq \exp( -E^4(x)/2 + x^9)
        \\ &
        \leq \exp( -E^4(x)/3).
        \end{split}
    \end{equation*}
    Again, $n_W + k +1 < 2 e^{2x}$, from which we deduce  
        $$
        \frac1{n_W +k +1} \log |Df^{n_W + k+ 1}(z)|  \ll -x.$$

        Let $X_x = \bigcup_{W\in \cW} (W \cap f^{-n_W}(S_W))$. Using a distortion bound of $2$, we obtain from the construction that 
        $$m(X_x) > m(A_*) \min_W \frac{m(S_W)}{4m(Q_W)} > \pi/8000$$
        and that for each $z \in X_x$, there is an $n$ with 
        $$
        \frac1n \log|Df^n(z)| >x,$$
        $$
        \frac1{n+1} \log|Df^{n+1}(z)| < -x.$$
        Necessarily, $f^{n+1}(z) \in B(0,1)$, so for some $C >0$ depending only on $\rho$,
        $$
        \frac1{n+1} \log|Df_\rho^{n+1}(z)| < -Cx.$$

        Taking a sequence of $x_j$ tending to $+\infty$, we obtain sets $X_{x_j}$ each with measure at least $\pi/8000$ and contained in the bounded set $B(0,1)$. Thus there  is a set $X_\infty$ of positive measure for which each $z \in X_\infty$ is in infinitely many of the $X_{x_j}$. Thus
        \eqref{eq:LEl1}, \eqref{eq:LEl2} hold for all $z \in X_\infty$, which implies \eqref{eq:LEl1}, \eqref{eq:LEl2} hold for all $z \in \bigcup_{n\geq 0} f^{-n}(X_\infty)$. 
        Using ergodicity and conservativity of Lebesgue measure (\cite{MeBartek}), $\bigcup_{n\geq0} f^{-n}(X_\infty)$ has full measure, completing the proof.
    \end{proof}
    Showing that the upper Lyapunov exponent is $0$ almost everywhere for the spherical metric is   more subtle. We need the following lemma. 
    
    Let $H : t \mapsto \exp(t^{1/10})$. For $t$ large enough, $H(t) > t$ and $H^2(t) > e^t$. 
    
    \begin{lem} \label{lem:spherest}
        Let $R >0$ be sufficiently large and let $Q \in \cQ$ be a subset of $\cL(-R)$ satisfying $|z| < 2|\re(z)|^2$ for all $z \in Q$. 
        Let $Z \subset \ccc$ and $n_Z \geq 0$ be such that $f^{n_Z}$ maps $Z$ bihilomorphically onto $Q$.
        There is a subset $Z_0 \subset Z$ and for each $z \in Z_0$ a number $n(z) \geq 1$ such that the following holds. 
        \begin{itemize}
            \item
        For $j = 1, \ldots, n(z)$, 
        $$
        \frac1j \log |D_\sigma f^j(f^{n_Z}(z))| < 1/ \log R;$$
            \item
        $m(Z \setminus Z_0)/m(Z) < 1/\log \log R$;
            \item
                if $z \in Z_0$, $f^{n_Z + n(z)}(z) \in \cL(-H(R))$; 
    \item
        $|f^{n_Z + n(z)}(z)| < 2|\re(f^{n_Z + n(z)}(z))|^2$;
            \item
                there is a finite partition of $Z_0$ into sets $U_i$ with associated numbers $n_i$, such that $n(z) = n_i$ for $z \in U_i$,  and such that $f^{n_Z + n_i}$ maps $U_i$ biholomorphically onto an element of $\cQ$. 
\end{itemize}
    \end{lem}

        \begin{proof}
            Let $y \geq R$ satisfy $\re(Q) = [-y -2\pi, -y)$. 
            Let $B_y = B(0,|\lambda_0| e^{-y})$, so $f(Q) \subset B_y$. 
            Let $\delta_0$ be given by Lemma~\ref{lem:Vballs}, so $0 < \delta_0 < \delta$. Let $n_Q$ be the maximal positive integer such that $f^j(B_y) \subset B(f^j(0), \delta_0)$ for $j = 0,1,\ldots, n_Q$. According to Lemma~\ref{lem:Vballs} then, a neighbourhood of $0$ is mapped biholomorphically onto $B(f^{n_Q}(0), \Delta \delta_0)$. Thus the distortion of $f^{n_Q}$ on $B_y$ is bounded by $2$. 
            Since $\diam(B_y)/2 = |\lambda_0|e^{-y} \geq |Df|$ on $Q$ and since $\delta_0 < \delta < 1/2$, it follows that 
            \begin{equation}\label{eq:df1plusj}
                |Df^{1+j}| < 1
            \end{equation}
                on $Q$ for $j = 0,\ldots, n_Q$. 
            Meanwhile, since $\delta_0 < \delta$, for $j = 0, \ldots, n_Q$ we have $f^j(B_y) \subset V \subset B(0,M)$, so the derivative at each step is bounded by $M$. 

            For $j < y/2 \log M$, 
            \begin{equation} \label{eq:Rover2}
                |Df^{j}| < e^{j \log M} < e^{y/2} \end{equation}
                on $B_y$. Thus  for $z \in Q$, for $j < y/2\log M$,
            \begin{equation} \label{eq:Rover4}
                \begin{split}
                    |D_\sigma f^{1+ j}(z)| & < \left({1+|z^2|}\right) |\lambda_0|e^{-y}e^{j \log M}  \\
                                           & < \left(1+|\re(z)|^4\right) |\lambda_0|e^{-y+ y/2}
                < (y+3\pi)^4 e^{-y/2} < e^{-y/3},
            \end{split}
            \end{equation}
                say. In particular, 
                for $z \in Q$ and $j = 1, \ldots, \lfloor y/2\log M \rfloor$, 
            \begin{equation} \label{eq:Rover5}
                \frac 1j \log |D_\sigma f^j(z)| < 0.
            \end{equation}
            This is our first estimate on the spherical derivative along the initial orbits of  points in $Q$. From \eqref{eq:df1plusj} we obtain, 
                for $z \in Q$ and $j = 1+ \lfloor y/2\log M \rfloor, \ldots, 1+n_Q$, 
            \begin{equation} \label{eq:Rover6}
                \frac 1j \log |D_\sigma f^j(z)| < \frac1j \log (1+|z|^2) < \frac{2 \log M}{y} \log(y+3\pi)^4 < y^{-1/2} \leq R^{-1/2}
            \end{equation}
            say. Combining \eqref{eq:Rover5} and \eqref{eq:Rover6} gives 
            \begin{equation} \label{eq:Rover7}
                \frac 1j \log |D_\sigma f^j(z)| < R^{-1/2}
            \end{equation}
            for all $z \in Q$ and $j = 1, \ldots, 1+n_Q$.

            Now we have to study what happens at times greater than $n_Q$. 
            By choice of $n_Q$,  we deduce $\diam(f^{n_Q}(B_y)) > \delta_0 /M$.
            Combined with \eqref{eq:Rover4}, it follows that $n_Q \geq y/2\log M$.
            It follows from the distortion bound that there is some $\nu_0 >0$, independent of $R, Q$, for which
            $m(f^{n_Q+1}(Q)) > \nu_0$. 
            Furthermore, 
             $f^{n_Q +1}(\partial Q)$ has length bounded by $10 \pi \delta_0 < 5\pi < 20$. 

             Let $x: =  y^{1/10}$.
             We claim that if $W_j, n_j$ for $j=1,2$ are such that $n_j$ is the first entry time of points in $W_j$ to $\cL(-2|\lambda_0|e^x)$, such that $f^{n_j}(W_j) \subset \cL(-e^{x+\sqrt{x}} +2\pi)$ and such that $f^{n_j}$ maps $W_j$ biholomorphically onto an element of $\cQ$, then $W_1$ and $W_2$ are pairwise disjoint. If $n_1 = n_2$, this is obvious since $\cQ$ is a partition. If $n_1 < n_2$, then $\diam(f^{n_1}(W_2)) < e^{-x}$ by Lemma~\ref{lem:Mdev2}, so $f^{n_1}(W_1) \cap f^{n_1}(W_2) = \emptyset$, by the first entry property, proving the claim.

            Set $A := f^{n_Q +1}(Q \setminus \partial Q)$, so $A$ is a simply-connected open set and, from above, $\partial A < 20$.  
            $C=20$ and 
            let $A_0 \subset A$ 
            be given by Proposition~\ref{lem:s0}, and for $z\in A_0$, let $k_0(z), k(z)$ be the numbers $n_0(z), n(z) \leq e^{3x}$ given by Proposition~\ref{lem:s0}. Then $k(z)$ is the first entry time of $z \in A_0$ to $\cL(-e^x)$ and $f^{k(z)}(z) \in \cL(-e^{x+\sqrt{x}})$. Let $W_z$ be the neighbourhood of $z$ mapped biholomorphically by $f^{k(z)}$ onto the element of $\cQ$ containing $f^{k(z)}(z)$, so $f^{k(z)}(W_z) \subset \cL(-e^{x+\sqrt{x}})$. Since $\dist(A_0, \partial A) \geq x^{-1/4}$, $W_z \subset A$. 
            
            By the claim, we obtain a cover of $A_0$ by a finite collection $\cW$ of pairwise-disjoint sets $W$ of the form $W_z, z \in A_0$. Extend the definition of $k_0, k$ to $z' \in W_z$ by $k_0(z') = k_0(z)$, $k(z') = k(z)$. Set $k_W = k(z)$ for $z \in W$. 
            On each $W$ the distortion of $f^j$ is bounded by $2$ for $j = 1, \ldots, k_W$ (as $P(f) \cap B(f^{k_W}(W), \Delta \diam(f^{k_W}(W)))  = \emptyset$). 
            Let us denote
            $$A' := \bigcup_{W \in \cW}  W.
            $$
            The measure estimate of Proposition~\ref{lem:s0} implies 
            \begin{equation} \label{eq:q111}
                m(A') > m(f^{n_Q+1}(Q)) - 1/\log x > m(f^{n_Q+1}(Q)) (1 - 1/ \nu_0 \log x).
            \end{equation}
            Let 
            $$
        Z_0 := Z \cap f^{-n_Q -1- n_Z}(A').
        $$
        The required partition of $Z_0$ is 
        $$\{Z \cap f^{-n_Q - 1-n_Z}(W) : W \in \cW\}.$$
        With the distortion of $f$ on $Q$ bounded by $e^{2\pi}$,  and distortion bounds of $2$ for $f^{n_Z}$ on $Z$ and for $f^{n_Q}$ on $f(Q)$, we derive from \eqref{eq:q111} that
        $$
        m(Z_0)/m(Z) >  1 - (4e^{2\pi})^2 / \nu_0 \log x > 1 - 1/ \log \log x^{10} \geq  1 - 1/\log \log R.$$

        Let $z \in Z_0$ and let $w := f^{n_Z}(z) \in Q$. Since $w \in Q$, in \eqref{eq:Rover7} we estimated $\frac1j \log |D_\sigma f^j((w))|$ for $j = 1, \ldots, n_Q + 1$, while $f^{n_Q +1}(w) \in B(0,M)$. Now we consider higher iterates. 
        For $j = n_Q +2, \ldots, 1 + n_Q + k_0(f^{1+n_Q}(w))$, we have the estimate 
        $|Df^j(w)| < 2e^{x^9}$
        coming from Proposition~\ref{lem:s0},  whence 
            \begin{equation} \label{eq:q211}
                \begin{split}
                    \frac1j \log |D_\sigma f^j(w)| &< \frac1{n_Q} \log ((1+M^2) 2e^{x^9}) < \frac{4 \log M }{ y} (y^{9/10} + \log(1+M^2)) \\
                                                   & < 5(\log M) y^{-1/10} \\
                    &< 1/\log R.
    \end{split}
            \end{equation}
            For $j =2 + n_Q + k_0(f^{1+n_Q}(w)), \ldots, 1+n_Q + k(f^{1+n_Q}(w))$, we have the estimate 
        $|Df^j(w)| < 2e^{x^9} |f^j(w)|^2$
        again coming from Proposition~\ref{lem:s0},  whence 
            \begin{equation} \label{eq:q311}
                \begin{split}
                    \frac1j \log |D_\sigma f^j(w)| &< \frac1{n_Q} \log \left(\frac{1+M^2}{1+ |f^j(w)|^2} 2e^{x^9} |f^j(w)|^2\right) \\
                &< \frac1{n_Q} \log ((1+M^2) 2e^{x^9}) \\
                    &< 1/\log R,
    \end{split}
            \end{equation}
            as before.
            
        Set $n(z) := n_Z + 1 + n_Q + k(f^{n_Z + 1+n_Q}((z)))$.  
            Combining \eqref{eq:Rover7}, \eqref{eq:q211} and \eqref{eq:q311} gives the required estimates on the spherical derivatives. 

         Once more from Proposition~\ref{lem:s0},   
         for $z \in Z_0$, 
         $$|f^{n(z)}(z)| < 2|\re( f^{n(z)}(z) )|^2, 
         $$
            and, since $e^x = e^{y^{1/10}} \geq H(R)$, 
         $$f^{n(z)}(z)  
            \in \cL(-e^{x}) \subset \cL(-H(R)),$$
            as required.
    \end{proof}
    \begin{lem} \label{lem:sphere}
        For almost every $z$ and the spherical metric $\sigma$, 
        $$
    \limsup_{n \to \infty} \frac{1}{n} \log |D_\sigma f^n(z)| = 0.$$
\end{lem}

\begin{proof}
            As before, by conservativity and ergodicity, we only need to show the result for a positive-measure set. 
            Let $R \gg 0$ and let $S \in \cQ$ with $S \subset \cL(-R)$. Let $E : t \mapsto e^t$. Repeatedly applying Lemma~\ref{lem:spherest}, in the limit we obtain a set $S_\infty$ for which
            \begin{equation*}
                \begin{split}
                    m(S_\infty)/m(S) & \geq \prod_{j = 0}^\infty \left(1 - \frac1{\log \log H^j(R)}\right) \\
                                     & \geq \prod_{j = 0}^\infty \left(1 - \frac1{\log \log H^{2j}(R)}\right)\left(1 - \frac1{\log \log H^{2j+1}(R)}\right) \\
                                     &> \prod_{j=0}^\infty \left(1 - \frac1{\log \log E^j(R)}\right)^2 \\
                                     &> 0,
                \end{split}
            \end{equation*}
            and for which, for each $z \in S_\infty$, 
            there is a strictly increasing sequence $n_j$, $j = 0, 1, \ldots$ such that
            $$
            \frac1k \log |D_\sigma f^k(f^{n_j}(z))| < \frac1{\log H^j(R)}$$
            for $k = 1, \ldots, n_{j+1} - n_j$. Consequently, for each $z$ in the positive-measure set $S_\infty$, 
        $$
    \limsup_{n \to \infty} \frac{1}{n} \log |D_\sigma f^n(z)| \leq 0.$$
\end{proof}

Theorems~\ref{thm:LE1} and~\ref{thm:LE2} follow immediately from Lemmas~\ref{lem:LEzero},~\ref{lem:LEinfty} and~\ref{lem:sphere}. 
\qed

        \section{Basic parametric estimates}
                We denote by $\log$ the principal branch of logarithm; it sends a neighbourhood of 1 in $\ccc$ to a neighbourhood of 0. In this section we commence our study of maps with parameters $\lambda$ in a neighbourhood of $\lambda_0$. 
            

        Let $z, \lambda_1, \lambda_2 \in \ccc$  and suppose $|\log(\lambda_1/\lambda_2)|$ is small. Let $g_i : z \mapsto \lambda_i e^z$ for $i=1,2$. 
        write $z_j := g_1^j(z)$ for $j \geq 0$. Suppose we have constructed $y_{k+1}, \ldots, y_n$ for some $0 \leq k < n$ and that 
        $1 - y_j/z_j$ is small for $j = k+1, \ldots, n$.
            We can formally set 
            \begin{equation}\label{eq:alf1}
            \alpha_j = \alpha_j(\lambda_1, \lambda_2, z) := \log(\lambda_1/\lambda_2) + \log(y_j/z_j) - (y_j-z_j)/z_j.
        \end{equation}
        While $|1 - y_j/z_j| < \frac12$, \eqref{eq:alf1} gives
            \begin{equation}\label{eq:alf3}
            |\alpha_j| < |\log (\lambda_1/\lambda_2)| + |(y_j - z_j)/z_j|^2.
        \end{equation}
            Set 
            \begin{equation}\label{eq:alf2}
            y_k := z_k + (y_{k+1} - z_{k+1})/z_{k+1} + \alpha_{k+1}, 
        \end{equation}
            so $g_2(y_k) = y_{k+1}$. 
            It follows that
            \begin{equation} \label{eq:pert}
            y_k - z_k = \frac{y_n-z_n}{Dg_1^{n-k}(z_k)} + \sum_{j=k+1}^{n} \frac{\alpha_j}{Dg_1^{j-k-1}(z_k)}.
        \end{equation}

        We shall use the above in Lemmas~\ref{lem:param2} and~\ref{lem:homoM}. The following proof just uses that $\lambda_1, \lambda_2$ are super-close and $n$ is not too big, while to prove Lemma~\ref{lem:homoM}, we use expansion to get summability in \eqref{eq:pert}.
        \begin{lem}\label{lem:param2}
            Let $x> 10$ and $c_0 \geq 1/e$.
            Let $\lambda_1, \lambda_2 \in \ccc \setminus \{0\}$  with $\beta := |\log(\lambda_1/\lambda_2)| < \exp(-9 c_0 e^{x})$, and let $g_i : z \mapsto \lambda_i e^z$ for $i=1,2$. 
            Let $n\leq e^{3x}$ and let $z = z_0 \in \ccc$. 
            Suppose that 
            $$
                    \inf_{j+k \leq n} |Dg_1^k(g_1^j(z))| > \exp(- 2c_0 e^{x}).$$
                Then there is a $y_0 = y(z, \lambda_1, \lambda_2, n)$ with 
                $g_2^n(y_0)=g_1^n(z)$ and, 
                for all $j \leq n$, 
                \begin{equation}\label{eq:g2g1}
                    |g_2^j(y_0) -g_1^j(z_0)| \leq \beta \exp(3 c_0 e^{x}) < \exp(-c_0 e^{x})
                \end{equation}
                Moreover, for all $j+k \leq n$, 
                \begin{equation}\label{eq:dg2g1}
                |\log Dg_2^k(g_2^j(y_0))/Dg_1^k(g_1^j(z_0))| <  \exp(-e^{x}).
                \end{equation}
            \end{lem}
        \begin{proof}
            The second inequality in \eqref{eq:g2g1} follows from the definition of $\beta$. 

            We commence by proving existence of $y_0$ satisfying \eqref{eq:g2g1} by induction on $n$. Write $z_j = g^j(z)$ for $j=0, \ldots, n$. So assume, for $j = 1, \ldots, n$, that there exists $y_j = y(z_j, \lambda_1, \lambda_2, n-j)$ satisfying $|y_j - z_j| \leq \beta \exp(3c_0 e^x)$ and, for $j = 1, \ldots, n-1$,  $g_2(y_j) = y_{j+1}$. 
            Existence of $y_n = z_n = y(z_n, \lambda_1, \lambda_2, 0)$ is trivial. 
            
            Define $y_0$ as per \eqref{eq:alf2}, so $g_2(y_0) = y_1$. 
            From \eqref{eq:pert} and the hypotheses on $n$ and the derivatives, one deduces for $k \geq 0$ that 
             $$|y_k - z_k| \leq e^{3x} \exp( 2c_0 e^{x}) \max_{j> k}|\alpha_j|.$$ For $k \geq 1$, $|z_k| > \exp(- 2c_0 e^{x})$ (by the derivative estimate), so
             \begin{equation}\label{eq:ykzkalph}
            |y_k - z_k|/|z_k| \leq e^{3x}\exp(4 c_0 e^{x}) \max_{j> k}|\alpha_j|.
        \end{equation}
        By \eqref{eq:ykzkalph} and \eqref{eq:alf3}, for $k \geq 1$, 
            $$
            |\alpha_k| < \beta + 3e^{6x}\exp(8 c_0 e^{x+1}) \max_{j>k}|\alpha_j|^2 < \beta + \beta^{-1} \max_{j>k}|\alpha_j|^2/4. 
            $$
            Now $|\alpha_n| = \beta$, so by induction it follows that $|\alpha_j| \leq 2\beta$ for $j =1, \ldots, n$. Hence $|y_0 - z_0| \leq 2\beta e^{3x} \exp(2c_0 e^{x}) \leq \beta\exp( 3c_0 e^{x})$. Thus $y_0$ satisfies \eqref{eq:g2g1}, completing the inductive argument. 

            To show \eqref{eq:dg2g1}, recall $|\alpha_l| \leq 2\beta$ and \eqref{eq:ykzkalph} and note that 
            \begin{equation*}
                \begin{split}
            \left|\log \frac{Dg_2^k(g^j_2(y_0))}{Dg_1^k(g^j_1(z_0))}\right| & = \left|\sum_{l=j+1}^{j+k} \log y_l/z_l \right| 
            \leq \sum_{l=j+1}^{j+k} 2\left| \frac{y_l - z_l}{z_l}\right| \\
            & \leq e^{3x} 4\beta e^{3x}\exp( 4c_0 e^{x}) \\ &< \exp(-e^x).
        \end{split}
    \end{equation*}
        \end{proof}
        Given a function $R: \ccc^2 \to \ccc$, for $j = 1, 2$ we let $D_jR(z_1,z_2)$ denote the partial derivative of $R$ with respect to the $j^{\mathrm{th}}$ variable, evaluated at the point $(z_1,z_2)$.
        \begin{lem} \label{lem:paraminv}
            Let $x> 10$.
            Let $B := \{\lambda \in \ccc : |\log (\lambda/\lambda_0)| <  \exp(-10 |\lambda_0| e^{x}) \}$. Suppose $U$ is a simply-connected open set. 
            Let  $n \leq e^{3x}$. Suppose for all $z \in U$ that 
            \begin{equation}\label{eq:jkc0x}
                    \inf_{j+k \leq n} |Df^j(f^k(z))| > 2\exp(- 2|\lambda_0| e^{x}).
                \end{equation}
                Then there is a holomorphic map $R : U \times B \to \ccc$ such that 
                \begin{equation}\label{eq:Rzl}
                    f^n_\lambda(R(z,\lambda)) = f^n(z)
                \end{equation}
                with 
                \begin{itemize}
            \item
                for $j =0, \ldots, n$, 
                \begin{equation} \label{eq:Riterates}
                    |f^j(z) - f^j_\lambda(R(z,\lambda))| < e^{-x};
                \end{equation}
                    \item
                $|D_1R(z,\lambda)| < \exp(-e^x);$ 
            \item
                $|D_2R(z,\lambda)| < \exp(4|\lambda_0| e^{x}).$ 
        \end{itemize}
            \end{lem}
        \begin{proof}
            Note that if $\lambda_1, \lambda_2 \in B$ then $|\log(\lambda_1/\lambda_2) | <2 \exp(-10 |\lambda_0| e^{x}) < \exp(-9|\lambda_0| e^{x})$. 
            With $c_0 = |\lambda_0|$, for each $z\in U$, $\lambda \in B$,  
            Lemma~\ref{lem:param2} spits out 
            a point $R(z,\lambda) := y(z, \lambda_0, \lambda, n)$ with $|f^j(z) - f^j_\lambda(R(z,\lambda))| < \exp(-|\lambda_0|e^x) < e^{-x}$ for $j=0,\ldots, n$. 
            We can immediately write $R(z,\lambda) = \phi_\lambda \circ f^n(z)$ where $\phi_\lambda$ is the appropriate inverse branch of $f^n_\lambda$, but it takes some work to show what appropriate is, and in particular that the branches vary continuously and so are well-defined.

            By \eqref{eq:g2g1}, 
            \begin{equation}\label{eq:fj1}
            |f^j_\lambda(R(z,\lambda)) - f^j(z)| < \exp(-|\lambda_0| e^{x}) < 1/2
        \end{equation}
            for $j=0, \ldots, n$. 
            Since $f_\lambda$ is univalent on each ball of radius $\pi$, $R(z,\lambda)$ is the unique point $z'$ with $f^n_\lambda(z') = f^n(z)$ for which $|f^j_\lambda(z') -f^j(z)| <1$ for all $j =0, \ldots, n$. 
            Now \eqref{eq:dg2g1} and \eqref{eq:jkc0x} imply
            $$
                    \inf_{j+k \leq n} |Df_\lambda^j(f_\lambda^k(R(z,\lambda))| > \exp(- 2|\lambda_0| e^{x}),
                    $$
                    so, for $\lambda' \in B$, we can apply Lemma~\ref{lem:param2} again to obtain  points
                    $y(R(z,\lambda), \lambda, \lambda',n)$. 
                    Again, for $j = 0, \ldots, n$,
                    $$|f^j_\lambda(R(z,\lambda)) - f^j_{\lambda'}(y(R(z,\lambda), \lambda, \lambda',n))| < \exp(-|\lambda_0| e^{x}) < 1/2$$
                    so with \eqref{eq:fj1}, the triangle inequality and uniqueness, one obtains $$y(R(z,\lambda), \lambda, \lambda',n) = R(z,\lambda').$$
                    The estimate \eqref{eq:g2g1} then implies that 
                    $$
                    |R(z,\lambda) - R(z,\lambda')| \leq |\log(\lambda/\lambda')| \exp(3|\lambda_0| e^{x})$$
                    so $R(z,\cdot)$ is continuous, with Lipschitz bound $\exp(4|\lambda_0| e^{x})$, say. Therefore the `appropriate' inverse branches $\phi_\lambda$ vary holomorphically, and $R(z,\cdot)$ is holomorphic with $|D_2 R(z,\lambda)| \leq \exp(4|\lambda_0| e^{x})$.

                    Differentiating \eqref{eq:Rzl} gives  $D_1R(z,\lambda) = Df^n(z) / Df^n_\lambda(R(z,\lambda))$, so \eqref{eq:dg2g1} implies
                    $$|\log D_1R(z,\lambda)| < \exp(-e^x),$$
                    and holomorphicity of $R$, as required. 
        \end{proof}

        The following lemma concerning existence of the holomorphic motion $h$ is well-known. We include the elementary proof for completeness, and because it gives the Lipschitz-type constant $M_0$ without invoking $\lambda$-lemmas. 
        \begin{lem} \label{lem:homoM}
            There exists $r_0, M_0 > 0$ and a function $h : P(f) \times B(\lambda_0, r_0)$ for which the following hold. For each $z \in P(f)$ and for $\lambda \in B(\lambda_0,r_0)$, $\lambda \mapsto h(z,\lambda)$ is holomorphic, while $z \mapsto h(z,\lambda)$ is injective, and  $|h(z, \lambda) - z| \leq M_0 |\lambda - \lambda_0|$. For such $z,\lambda$ and all $n \geq 0$, 
            \begin{equation}\label{eq:hrel}
                f^n_\lambda ( h(z,\lambda)) = h( f^n (z), \lambda).
            \end{equation}
        \end{lem}
        \begin{proof}
            Note that if \eqref{eq:hrel} holds with $n=1$ then it holds for all $n\geq 0$. 

            Since $P(f)$ is a compact, forward-invariant, hyperbolic repelling set, there is a  constant $M_1> 1$ such that $\sum_{j\geq 1} |Df^j(z)|^{-1} < M_1$ for all $z \in P(f)$, and there is an $\eta \in (0,1)$ such that $B(0,\eta) \cap P(f) = \{0\}$. 
            Choose $r_0>0$ such that,
for all $\lambda \in B(\lambda_0, r_0)$, 
            $$ r_\lambda : = \max(|\log(\lambda/\lambda_0)|, |\lambda - \lambda_0|) <  \eta^2/4M_1^2.$$ 

            As an intermediate step, we shall inductively construct functions $h_n$ which shall converge to $h$. 
            Let $h_0 : (z,\lambda) \mapsto z$ and suppose for $j = 1, \ldots, n-1$ we have functions $h_j : P(f) \times B(\lambda_0, r_0) \to \ccc$ such that, for all $(z, \lambda) \in P(f) \times B(\lambda_0,r_0)$, 
            \begin{itemize}
                \item
            $h_{j-1}(f(z),\lambda) = f_\lambda(h_j(z, \lambda))$; 
                \item
            $|h_j(z, \lambda) - z | \leq 2 M_1 r_\lambda $. 
            \end{itemize}

            Then for each such pair $(z,\lambda)$ we have the sequences $z = z_0, z_1 = f(z), \ldots, z_n = f^n(z)$ and $y_1 = h_{n-1}(z_1, \lambda), \ldots, y_n = z_n$ and the corresponding sequence of $\alpha_j = \alpha_j(\lambda_0, \lambda, z)$ as defined in \eqref{eq:alf1}. Then define $y_0$ by \eqref{eq:alf2}, whence $f_\lambda(y_0) = y_1$.  
            For $j \geq 1$, by supposition, $|y_j - z_j| \leq 2M_1 r_\lambda$, while $z_j \in P(f) \setminus \{0\}$  so $|z_j|\geq\eta$. 
In particular, $|(y_j - z_j)/z_j| \leq 2 M_1 r_\lambda/\eta.$ Inserting this estimate into \eqref{eq:alf3}, we obtain
            $$|\alpha_j| \leq |\log(\lambda/\lambda_0)|  + 4M_1^2 r_\lambda^2/\eta^2 \leq 2 r_\lambda.$$  
            By (\ref{eq:pert}) and the definition of $M_1$, we deduce that $|y_0 - z_0| \leq 2 M_1r_\lambda.$ 
            Define $h_n(z,\lambda) := y_0$. Then 
            \begin{equation} \label{eq:homo}
                h_{n-1}(f(z),\lambda) = f_\lambda(h_n(z,\lambda)) \text{ and }
            |h_n(z, \lambda) - z | \leq 2 M_1 r_\lambda.
        \end{equation}
            To conclude the inductive construction of  $h_n$, note that a $h_1$ clearly exists satisfying the required properties. 
            Thus \eqref{eq:homo} holds for each $n$.

            Consequently $|h_{n-1}(f(z),\lambda) - f(z)| \leq 2 M_1 r_\lambda < \eta$, while $|f(z)| \geq \eta$, so $h_{n-1}(f(z),\lambda) \ne 0$ and 
            $$
            \lambda \mapsto h_n(z,\lambda) = f_\lambda^{-1}(h_{n-1}(f(z),\lambda))$$
            is well-defined and holomorphic, upon choosing the appropriate branch of $f_\lambda^{-1}$. 

            Since the $h_n(z,\cdot)$ are uniformly bounded, we can extract a convergent subsequence with holomorphic limit $h(z, \cdot)$ with the same Lipschitz bound $|h(z, \lambda) - z| \leq 2M_1r_\lambda$. One can take $M_0 := 2M_1$. The map $h$ satisfies \eqref{eq:hrel} for $n=1$ and thus for all $n$. 
            We claim that, for given $\lambda$,  $h(z, \lambda)$ is the unique point $z_\lambda$ such that $|f^n_\lambda(z_\lambda) - f^n(z)| < \delta$ for all $n \geq 0$. Now $f^{n_0}_\lambda$ is uniformly expanding on $B(f^n(z), 3\delta)$ for each $n$. Therefore there is only one point, $z'$,  for which $f^n_\lambda(z') \in B(f^n_\lambda(z_\lambda), 2\delta)$ for all $n\geq 0$ and $z'= z_\lambda$, proving the claim.  Therefore the map $h$ is unique and $z \mapsto h(z,\lambda)$ is injective. 
        \end{proof}

        \section{Parameter space to phase space near $P(f)$}
The following lemma  is another form of the standard Koebe distortion lemma.
	\begin{lem}\label{lem:K2}
	Given $\eps' >0$ there is a $\delta' >0$ such that if $g$ is any univalent function on the unit disc, one can write
	$$Dg(z) = Dg(0)[1 + \theta(z)],$$
	where $\theta$ is a holomorphic function on $B(0,\delta')$ with $|\theta| < \eps'$. 
	\end{lem}
	\begin{proof}
		The distortion of $g$ is bounded by $2$ on $B(0, 1/\Delta)$, so $|Dg(z)| \leq 2|Dg(0)|$ on that ball. By Cauchy's integral formula, $|D^2g| \leq 4\Delta |Dg(0)|$ on $B(0, 1/2\Delta)$. 
	Integrating gives $|Dg(z) -Dg(0)| \leq 4|z|\Delta |Dg(0)|,$ on $B(0, 1/2\Delta)$. Taking $\delta' = \eps' /4\Delta$,  the result follows.
	\end{proof}

	The ideas in this section are not especially new, though the exposition and the formulation of results are. The reader may wish to compare this section with  \cite[Sections~3,~4]{Bad:Rare} and \cite[Section~3]{Asp:Rare}. The useful result is Lemma~\ref{lem:xiinject}; it follows easily from the following proposition. 

        Recall the definitions of Section~\ref{sec:glob}. Let $h, M_0, r_0$ be given by Lemma~\ref{lem:homoM}. 
        Now $h(0, \cdot)$ is a holomorphic function of $\lambda$. \emph{A priori} it could be identically zero, however Misiurewicz maps are not structurally stable (\cite{Makienko:Exp, UZ:ExpInstability}), so $h\not \equiv 0$, see \cite[Lemma~2.1]{Bad:Rare}. 
            Therefore, there exist an integer $K \geq 1$ and a non-zero constant $a_K$ such that $h(0,\lambda) = a_K(\lambda-\lambda_0)^K + \text{higher order terms}$. 
            Thus given  $\eps_1 \in (0,1)$, there is an $r(\eps_1) >0$ for which we can write
            \begin{equation} \label{eq:hlam}
            h(0, \lambda) = a_K(\lambda-\lambda_0)^K[1+ \theta_0(\lambda)],
        \end{equation}
            where $\theta_0$ is holomorphic on $B(\lambda_0, r(\eps_1))$ with norm bounded by $\eps_1$.
            In particular, for $\lambda \in B(\lambda_0, r(\eps_1))$, 
            \begin{equation} \label{eq:hexp}
        \left|a_K(\lambda -\lambda_0)^{K}\right|/2 \leq 
            \left|h(0,\lambda)\right| \leq 2\left|a_K(\lambda -\lambda_0)^{K}\right|.
        \end{equation}

        For $n\geq0$, let us denote by $\xi_n$ the holomorphic map defined by
        $$
        \xi_n(\lambda) = f^n_\lambda(0).$$
        \begin{prop} \label{prop:xigrows}
            Given $\eps > 0$, there exist constants $\delta_1, r_3, C_0, C_1 >0$ such that, for all $r \in (0, r_3)$, the following holds.  
        Let $n = n(r, \delta_1)$ be maximal such that 
            \begin{equation} \label{eq:ndef}
        f_\lambda^j(B(0, 2|h(0,\lambda)| )) \subset B(f^j(0), \delta_1) \subset V
        \end{equation}
        for $j = 0, \ldots, n$ and all $\lambda \in B(\lambda_0, 2r)$. 

        Then $\xi_n(B(\lambda_0, 2r)) \subset B(f^n(0), \delta_1)$, 
        $$
        D\xi_n(\lambda) = -Df^n(0) Ka_K(\lambda - \lambda_0)^{K-1} \left[1+\theta_5(\lambda)\right],
        $$
        where $\theta_5$ is a holomorphic function on the annulus $A(\lambda_0; r/4, r)$ with  $|\theta_5| < \eps$,  and
        $$
        1/C_0 r < |D\xi_n(\lambda)| < C_0/r.$$
             Moreover, $|Df^n_\lambda(0)| \leq C_1/r^K$ for all $\lambda \in B(\lambda_0, r)$. 
    \end{prop}
    \begin{proof}
        Taking $\delta_1 < \delta$, $B(f^j(0), \delta_1) \subset V$. From \eqref{eq:ndef}, 
        the statment $\xi_n(B(\lambda_0, 2r)) \subset B(f^n(0), \delta_1)$ is trivial.
        We shall expend much effort to compare $Df_\lambda^n(z), Df^n_\lambda(0)$ and $Df^n(0)$. 

        Assume $\eps \in (0,1)$ and set $\eps_1 = \eps/16$. 
            Let $\delta'$ be given by Lemma~\ref{lem:K2} for $\eps' = \eps_1$  and let 
         $\delta_1 \in (0,\min(\delta_0\delta', \delta)/2)$ satisfy 
            \begin{equation} \label{eq:deldef}
         \delta_1 e^{Mn_0 +  \alpha } \sum_{k \geq 0} e^{-k\alpha/n_0} < \eps_1/8.  
        \end{equation}
        Let $r_1$ be the number $r(\eps_1)>0$ for which \eqref{eq:hlam} holds.
         Let $r$ satisfy
            \begin{equation} \label{eq:rdef}
        0 < r <\min(\eps_0, r_0, r_1,  \delta_1/M_0)/2 
        \end{equation}
        and let $n = n(r, \delta_1)$ be 
        given by \eqref{eq:ndef}.

        By \eqref{eq:ndef} and choice of $\delta_1$, 
        \begin{equation} \label{eq:qwe}
	f_\lambda^n(B(0, 2|h(0,\lambda)|)) \subset B(f_\lambda^n(0), 2\delta_1)  \subset B(f^n_\lambda(0), \delta_0).
	\end{equation}
        By Lemma~\ref{lem:Vballs}, 
        a neighbourhood of $0$ is mapped biholomorphically by $f_\lambda^n$ onto $B(f_\lambda^n(0), \delta_0)$. By \eqref{eq:qwe}, this neighbourhood necessarily contains  
        $B(0, 2|h_\lambda(0,\lambda)|)$, and 
        $\delta_0/2\delta_1 \geq \delta'$ plus choice of $\delta'$ then implies 
            \begin{equation} \label{eq:reps}
                Df^n_\lambda (z) = Df^n_\lambda(0) [1 + \gamma(z, \lambda)],
        \end{equation}
        where $\gamma(\cdot, \lambda)$ is a holomorphic function on $B(0, 2|h(0,\lambda)|)$  bounded by $\eps_1$, and this for each $\lambda \in B(\lambda_0,2r)$.

        Integrating along a ray from $0$ to $z$, we obtain
            \begin{equation} \label{eq:reps2}
                f^n_\lambda (z) = f^n_\lambda(0) + z Df^n_\lambda(0)[1 +  \gamma_1(z,\lambda)],
        \end{equation}
                where $\gamma_1(z,\lambda) := \frac1z \int_0^z \gamma(w, \lambda)$,
        with $\left| \gamma_1(z,\lambda) \right| \leq \eps_1$. 

        Applying  \eqref{eq:reps2}, with $z = h(0,\lambda)$, gives
            \begin{equation} \label{eq:reps2bis}
         f^n_\lambda(0) - f^n_\lambda(h(0,\lambda)) = -Df^n_\lambda(0) h(0,\lambda) \left[1+ \theta(\lambda) \right],
        \end{equation}
        where $\theta$ is the holomorphic function $\lambda \mapsto \theta(\lambda) := \gamma_1(h(0,\lambda), \lambda)$ with norm bounded by $\eps_1$. 

    Now we wish to compare $Df^n_\lambda$ with $Df^n$ at $0$. First we show $n$ is not too large.

    By \eqref{eq:hexp}, there is a $\lambda_1 \in B(\lambda_0, 2r)$ for which $|h(0,\lambda_1)| > |a_K| r^K$. Since $|Df^{n_0}|>\exp(\alpha)$ on $V$, if $kn_0 \leq n$ then 
    $$
    B(f^{kn_0}(0), 2\delta_1) \supset f^{kn_0}(B(0, 2|h(0,\lambda_1)|)) \supset B(f^{kn_0}(0), e^{k\alpha} a_Kr^K) .
    $$
    Thus ${\alpha n/n_0} \leq \log (2\delta_1 r^{-K}/|a_K|)$. In particular, there exists a $c_0 >0$  for which
    $$ n = n(r, \delta_1) < - c_0 \log r.$$
    This implies that $rn(r,\delta_1) \to 0$ as $r \to 0$. 
    
            Recall $|\lambda_0|\geq \frac1e$ and $V \subset B(0,M-2)$, so $|Df_\lambda| \geq e^{-M}$ on $V$ for all $\lambda \in B(\lambda_0, 1/2e)$. 
            By the same Koebe distortion bound that gave \eqref{eq:reps}, and the estimates 
            $$|Df_\lambda^k(f_\lambda^j(0))| \geq e^{-n_0M} \exp(\lfloor k/n_0 \rfloor \alpha)$$
            for $k+j = n$, we deduce that the images of $B(0,2|h(0,\lambda)|)$ under $f^j$ are exponentially small in $n-j$: 
            \begin{equation} \label{eq:rdiam}
            \diam(f^j(B(0,2|h(0,\lambda)|))) \leq 2 \delta_1 e^{n_0M + \alpha}e^{(j-n)\alpha/n_0}.
        \end{equation}
        Meanwhile, by definition of $h$, for all $j \geq 0$, 
        $$h(f^j(0),\lambda) = f^j_\lambda(h(0,\lambda)) \in f^j_\lambda(B(0, 2|h(0,\lambda)|)),$$ 
        while $|h(z,\lambda) - z| \leq M_0|\lambda - \lambda_0|.$
        Hence 
            \begin{equation} \label{eq:hdiam}
        \dist(f^j(0), f_\lambda^j(B(0,2|h(0,\lambda)|))) 
        \leq M_0|\lambda - \lambda_0|.
        \end{equation}

        For $j \leq n$, combining \eqref{eq:hdiam} and \eqref{eq:rdiam} gives
        $$|f^j(0) - f^j_\lambda(0)| 
        \leq M_0 |\lambda - \lambda_0| + 2 \delta_1 e^{n_0M + \alpha}e^{(j-n)\alpha/n_0}.$$
        As an exponential map, $Df(y)/Df(y') = e^{y-y'}$. 
        By \eqref{eq:deldef}, there is a uniform bound  
        \begin{equation} \label{eq:fjfj}
            \sum_{j=0}^{n-1} |\log |Df(f^j(0))/Df(f^j_\lambda(0))|| \leq \sum_{j=0}^{n-1} 
        |f^j(0) - f^j_\lambda(0)| 
            <2  M_0 r n 
        + \eps_1/4,
    \end{equation}
        while $Df/Df_\lambda = \lambda_0/\lambda$. 
        Thus for $k \leq n$, 
        \begin{equation} \label{eq:fkfk}
            \begin{split}
        |\log |Df^k(0)/Df^k_\lambda(0)|| &< 2 M_0 rn +  \eps_1/4 + |n \log (\lambda_0/\lambda) | \\
        &< 2M_0 rn +  \eps_1/4+ 2 n |\lambda - \lambda_0|/|\lambda_0|. \\
    \end{split}
\end{equation}
        But from above, $rn \to 0$. Thus if  $r$ is sufficiently small, 
        $$
        |\log |Df^n(0)/Df^n_\lambda(0)|| < \eps_1/2$$
        so
        \begin{equation} \label{eq:thetal}
            Df^n_\lambda(0) = Df^n(0)[1 + \theta_1(\lambda)]
        \end{equation}
        with $\theta_1$  a holomorphic function on $B(\lambda_0, 2r)$ with norm bounded by $\eps_1$. 
        With \eqref{eq:reps2bis}, we obtain
            \begin{equation} \label{eq:reps3bis}
         f^n_\lambda(0) - f^n_\lambda(h(0,\lambda)) = -Df^n(0) h(0,\lambda) \left[1+ \theta_2( \lambda) \right],
        \end{equation}
        where $\theta_2 = (1+\theta)(1+\theta_1)$ is a holomorphic function with norm bounded by $3\eps_1$. 

        Using \eqref{eq:hlam}, we can substitute in for $h$ to obtain
            \begin{equation} \label{eq:reps9}
         f^n_\lambda(0) - f^n_\lambda(h(0,\lambda)) = -Df^n(0) a_K (\lambda-\lambda_0)^K \left[1+ \theta_3(\lambda) \right],
        \end{equation}
        where $\theta_3 : = (1+\theta_2)(1+\theta_0)$ is holomorphic with norm bounded by $5\eps_1$ on $B(\lambda_0, 2r)$. 
        By Cauchy's integral formula,  $|D\theta_3(\lambda)| < 10 \eps_1/r$ on $B(\lambda_0, r)$, whence $|\lambda -\lambda_0| |D\theta_3(\lambda)| < 10\eps_1$. 
        Thus, on $B(\lambda_0,r)$, the derivative of \eqref{eq:reps9} can be written
            \begin{equation} \label{eq:reps8}
                -Df^n(0) K a_K (\lambda-\lambda_0)^{K-1} \left[1+ \theta_4(\lambda) \right],
        \end{equation}
        where $1 + \theta_4(\lambda) := (1+\theta_3(\lambda)) +(\lambda-\lambda_0) D\theta_3(\lambda)/K$, so $|\theta_4(\lambda)| < 15\eps_1.$

        Now we have all the distortion-like estimates we need, let us estimate the size of the derivative. 
        By maximality of $n$, there exists $\lambda_1 \in B(\lambda_0, 2r)$ for which 
        $f_{\lambda_1}^{n+1}(B(0, 2|h(0,\lambda_1) | )) \not\subset B(f^{n+1}(0), \delta_1)$, which, combined with \eqref{eq:hdiam} implies
            \begin{equation} \label{eq:hdiam2}
                \diam(f_{\lambda_1}^{n+1}(B(0,2|h(0,\lambda_1) |))) 
        \geq \delta_1 - M_0|\lambda_1 - \lambda_0| > \delta_1 - M_0r > \delta_1/2. 
        \end{equation}
        The derivative is bounded by $M$ on $V$, so \eqref{eq:hdiam2} implies
            \begin{equation} \label{eq:hdiam3}
                \diam(f_{\lambda_1}^{n}(B(0,2|h(0,\lambda_1) |))) 
        \geq \delta_1 /2M.
        \end{equation}
        Therefore, for some $z \in B(0,2|h(0,\lambda_1 )|)$,
            \begin{equation} \label{eq:hdiam4}
                |Df_{\lambda_1}^{n}(z)| 
                \geq \frac{\delta_1}{ 4M |h(0,\lambda_1) | }.
        \end{equation}
        The bounds \eqref{eq:thetal} and \eqref{eq:reps} give good distortion control, combining to give $Df^n_\lambda(z) = Df^n(0) [(1+\theta(\lambda)) (1 + \gamma(z, \lambda))]$, so \eqref{eq:hdiam4}
        implies $$|Df^n(0)| |h(0,\lambda_1)| > \delta_1/8M,$$
        in turn implying, via \eqref{eq:hexp},
            \begin{equation} \label{eq:dfnbound}
        |Df^n(0)| |a_K| (2r)^K > \delta_1/16M.
        \end{equation}
        If $|\lambda-\lambda_0| \geq r/4$ then 
            \begin{equation} \label{eq:lrbound}
        \frac{|\lambda-\lambda_0|^{K-1}}{r^K} \geq \frac1{4^{K-1}r}. 
        \end{equation}
        From \eqref{eq:dfnbound} and \eqref{eq:lrbound}, we  deduce that, on the annulus $A(\lambda_0; r/4, 2r)$,  
            \begin{equation} \label{eq:dfnbd2}
                |Df^n(0)| K|a_K||\lambda - \lambda_0|^{K-1}|1+\theta_4(\lambda)|   > K  \delta_1/2^{8K}M r.
        \end{equation}

        Now $\xi_n(\lambda) = f^n_\lambda(0)$, so adding and subtracting the same term, 
        $$
        \xi_n(\lambda) = f^n_\lambda(0) - f^n_\lambda(h(0,\lambda))  + h(f^n(0), \lambda) ,$$
        and  \eqref{eq:reps9} gives, on $B(\lambda_0, 2r)$, 
            \begin{equation} \label{eq:xi1}
        \xi_n(\lambda) =  -Df^n(0) a_K(\lambda-\lambda_0)^K \left[1+ \theta_3(\lambda) \right] + h(f^n(0), \lambda).
        \end{equation}
       Let $D_2h$ denote the partial derivative of $h$ with respect to the second variable.
       Taking the derivative on both sides of \eqref{eq:xi1}, and using \eqref{eq:reps8}, 
            \begin{equation} \label{eq:xi2}
                    D\xi_n(\lambda) =  Df^n(0) K a_K(\lambda-\lambda_0)^{K-1}\left[1+\theta_4(\lambda) \right] 
        + D_2h(f^n(0), \lambda).
        \end{equation}
        Now  $|h(z,\lambda) -z| \leq M_0 |\lambda - \lambda_0|$ for $z \in P(f)$ and $\lambda \in B(\lambda_0, r_0)$, so by Cauchy's integral formula, $|D_2 h(z, \lambda)| \leq 2M_0$ on $B(\lambda_0, r)$. 
        Therefore, if $r$ is small enough the bound \eqref{eq:dfnbd2} together with \eqref{eq:xi2} entails that
            \begin{equation} \label{eq:xi4}
        D\xi_n(\lambda) = -Df^n(0) Ka_K(\lambda - \lambda_0)^{K-1} \left[1+\theta_5(\lambda)\right],
        \end{equation}
        where $\theta_5$ is a holomorphic function on $A(\lambda_0; r/4, r)$ with norm bounded by $16\eps_1$. Setting $C_0 := 2^{9K}M/K \delta_1$, taking absolute values of \eqref{eq:xi4} and using \eqref{eq:dfnbound}, we obtain
        $$
        |D\xi_n(\lambda)| > 1/C_0 r.$$

        It remains to provide the upper bound for $|Df^n_\lambda(0)|$. This follows simply from \eqref{eq:dfnbound} and \eqref{eq:thetal}.
    \end{proof}

        \bigskip

        \begin{lem} \label{lem:inject}
            Let $g$ be a holomorphic map defined on an open convex set $U$. Suppose $\re(Dg(z)) > 0$ for all $z\in U$. Then $g$ is injective. 
        \end{lem}
        \begin{proof}
            Integrating $Dg$ along a line from $z_1$ to $z_2$ in $U$, one cannot obtain $0$. 
        \end{proof}

         Given an annulus 
         $A(y; a_1,a_2)$   
         and $k\geq 2$, the $k$ rays leaving $y$ with angles $2j\pi /k$ for $j \leq k$ divide $A(y;a_1,a_2)$ into $k$ (open) congruent pieces which we will call \emph{$k$-sectors} of $A(y;a_1,a_2)$.

         \begin{lem} \label{lem:xiinject}
             Given $\eps' >0$, there exists $r_3, \gamma \in (0,1)$ and $\nu_0, C, C_0, C_1>0$ such that for all $r \in (0, r_3)$, the following holds. 
             There exists $n \geq 1$ such that $\xi_n$ maps each $4K$-sector $W$ of $A(\lambda_0; \gamma r, r)$ injectively onto a simply-connected, open set $\xi_n(W)$ with $m(\xi_n(W) ) > \nu_0$ and the length of $\partial \xi_n(W)$ bounded by $C$. For $j \leq n$, $\xi_j(W) \in V$. 
             
             For $\lambda,\lambda' \in W$,
             $$
             1/C_0 r \leq |D\xi_n(\lambda)| 
             $$
             and 
             $$
             \left| \frac{D\xi_n(\lambda)}{D\xi_n(\lambda')} \right| < 1+ \eps'.$$ 

             Moreover, $|Df^n_\lambda(0)| \leq C_1/r^K$ for all $\lambda \in B(\lambda_0, r)$. 
         \end{lem}
         \begin{proof}
             Let $\gamma <1$ satisfy $\gamma^K > 1-\eps'/3$. 
             With $\eps = \eps'/3$, let $\delta_1, r_3, C_0, C_1, \theta_5$ be given by Proposition~\ref{prop:xigrows}, let $r \in (0,r_3)$ and let $n$ be defined as per Proposition~\ref{prop:xigrows}. Then $\xi_j(B(\lambda_0, r)) \subset V$ for $j \leq n$.  
             
             Let $\gamma \in (\frac12, 1)$ and let $W$ be a $4K$-sector of $A(\lambda_0; \gamma r, r)$. Let $\widehat W$ denote the convex hull of $W$, so $\widehat W$ is contained in a $4K$-sector $W'$ of $A(\lambda_0; r/4, r)$. Now 
             $$\{(\lambda - \lambda_0)^{K-1} : \lambda \in W'\}$$ 
             lies (strictly) in a quadrant of the plane. Since $|\theta_5|< |\eps| < 1/\sqrt{2}$ on $A(\lambda_0; r/4, r)$, 
             $$\{ 1+ \theta_5(\lambda) : \lambda \in W'\}$$
             is also a subset of a quadrant. Thus 
             \begin{equation}\label{eq:lll}
        D\xi_n(\lambda) = -Df^n(0) Ka_K(\lambda - \lambda_0)^{K-1} \left[1+\theta_5(\lambda)\right]
    \end{equation}
            lies in a fixed half-plane for all $\lambda \in \widehat W$. By Lemma~\ref{lem:inject}, $\xi_n$ is injective on $\widehat W$ and thus is injective on $W$. 

            The derivative estimate $|D\xi_n(\lambda)| > 1/ C_0 r$ on $W$ implies the image has measure at least $\nu_0 $, for some $\nu_0 >0$ depending on $\gamma$ but not on $r$. Injectivity and bounded distortion  give an upper bound on $r |D\xi_n|$, since the measure of $V$ is bounded.  The length of $\partial W$ is bounded by a constant times $r$, so the upper bound on $r|D\xi_n|$ implies that the length of $\partial \xi_n(W)$ is bounded by a constant $C>0$. 

            The distortion estimate follows from $\eqref{eq:lll}$, as choice of $\gamma$ and the bound $|\theta_5| < \eps'/3$ give 
            $$
             \left| \frac{D\xi_n(\lambda)}{D\xi_n(\lambda')} \right| <  
            1/(1-\eps'/3)^2 < 1+\eps'.$$ 

            The derivative estimate of $|Df^n_\lambda(0)| \leq C_1/r^K$ comes directly from Proposition~\ref{prop:xigrows}.
        \end{proof}

        \section{Parameter dependence at the large scale}

        Lemma~\ref{lem:paraminv} allows us to show that some sets which get mapped eventually onto a square far out to the left do not move very fast as the parameter $\lambda$ varies, so if $\lambda$ does not vary much, the intersection remains large. Later on we will show that for relatively large sets of parameters, the orbit of 0 under $f_\lambda$ lands in one of these intersections. 

            \begin{lem} \label{lem:overlaps}
                Let $C, \nu_0>0$. 
                There is an $M_2>0$ such that for $x>M_2$, the following holds. 
            Suppose $A \subset B(P(f), 1)$ is a simply-connected open set satisfying $m(A) > \nu_0$ and with $\partial A$ having length at most $C$. 
            Let $B :=\{ \lambda : |\log(\lambda_0/\lambda)| < \exp(-10|\lambda_0|e^{x})\}$.  There is a collection $\{U_l\}_{l=1}^L$ of pairwise-disjoint subsets of $A$  and numbers $n_l$, together with a map $R : \bigcup_l U_l \times B \to A \setminus B(\partial A, e^{-x})$ such that
            \begin{itemize}
                 \item
                    $m(\bigcup_l U_l)/m(A) \geq  1 - 1/\log \log x$;
                \item 
                    $R(z,\lambda_0) = z$;
                \item
                    on each $U_l \times B$, 
                     $R$ is holomorphic,  $|\log D_1R| < \exp(-e^x)$ and  $|D_2R| < \exp(4 |\lambda_0| e^{x})$;
                \item
		for $z \in R(U_l, \lambda)$, 
                    $$|Df_\lambda^{n_l}(z)|< 3e^{x^9}|\Re(f^{n_l}(z))|^4;$$
                \item
                    for each $\lambda$, the sets $R(U_l, \lambda)$ for $l =1, \ldots, L$ are pairwise-disjoint;
		    \item
		for $z \in R(U_l, \lambda)$, a neighbourhood $V_z$ of $z$ with diameter bounded by $e^{-x}$ gets mapped biholomorphically by $f^{n_l}_\lambda$ onto 
		$$B(f^{n_l}_\lambda(z), 1) \subset \cL(-e^{x+\sqrt{x}} + 3\pi).$$
                \end{itemize}
            \end{lem}
            \begin{proof}
                Given $C, \nu_0 >0$, let $x\gg 0$ be large enough to apply 
             Proposition~\ref{lem:s0}. 
             Let $A_0 \subset A \setminus B(\partial A, x^{-1/4})$ and $n(z)$ for $z \in A_0$ be given by Proposition~\ref{lem:s0}, so $m(A\setminus A_0)  \leq 1/\log x$ and $n(z) \leq e^{3x}$. 

                For $z \in A_0$, let $Q_z$ be the element of $\cQ$ containing $f^{n(z)}(z)$. 
                Let $U_z$ be the neighbourhood of $z$ mapped biholomorphically by $f^{n(z)}$ onto $Q_z$. 
                Clearly $Q_z \subset \cL(-e^{x +\sqrt{x}} + 2\pi)$, and for $j< n(z)$, the diameter of $f^j(U_z)$ is bounded by $e^{-x}$ (see Lemma~\ref{lem:Mdev2} to treat $j \leq n(z)-2$, while $|Df| \geq e^{x +\sqrt{x}}$ on $f^{n(z)-1}(U_z)$). 
                Since $n(z)$ is also the first entry time of $z$ to $\cL(-2|\lambda_0|e^x)$, $f^j(U_z) \subset \cR(-2|\lambda_0|e^x -1)$ for $j < n(z)$. It follows that if $z' \in A_0$ and $U_z  \cap U_{z'} \ne \emptyset$  then $n(z) = n(z')$ and $U_z = U_{z'}$. 
                Thus the neighbourhoods $U_z$, for $z \in A_0$,  form a finite (since $n(z)$ is bounded), pairwise-disjoint collection which we can write as $\{U_l\}_{l=1}^L$, setting $n_l := n(z)$ for some  $z \in U_l \cap A_0$. The collection is a cover of $A_0$ and thus has measure at least $m(A) - 1/\log x$. Since $m(A) > \nu_0$, for large $x$ we obtain the required measure estimate. 
		
		We can write $Q_l = f^{n_l}(U_l) \in \cQ$. 
                 Let $\widehat{U}_l \supset U_l$ denote the set containing $U_l$ mapped biholomorphically by $f^{n_l}$ onto $B(Q_l, 1)\subset \cL(-e^{x+\sqrt{x}} + 3\pi).$  
		Applying Lemma~\ref{lem:distnQ}, 
                the distortion of $f^{n_l}$ is bounded by $2$ on each $\widehat{U}_l$, and since $A_0 \cap \widehat{U}_l \ne \emptyset$, the estimates of Proposition~\ref{lem:s0} imply that for $z \in \widehat{U}_l$, 
            \begin{equation} \label{eq:dful}
                |Df^{n_l}(z)| < 2e^{x^9} \sup_{y\in \widehat{U}_l} |\Re(f^{n_l}(y))|^4 < 2 e^{x^9}  |\Re(f^{n_l}(z)) +2\pi|^4 
            \end{equation}
	    and
	    $$
	    \inf_{j+k \leq n_l} |Df^j(f^k(z))|  > 2\exp(-2|\lambda_0| e^x).$$
                 We can therefore apply  Lemma~\ref{lem:paraminv} to obtain a holomorphic map
            $R_l : \widehat{U}_l \times B \to \ccc$, 
            where $R_l(z, \lambda_0) = z$ and, for $(z,\lambda) \in \widehat{U}_l\times B$,  
            $$
            f_\lambda^{-n_l} \circ R_l (z, \lambda) = f^{n_l}(z).
            $$
            By Lemma~\ref{lem:paraminv},  $|\log D_1R_l| < \exp(-e^x)$ and $|D_2R_l(z,\lambda)| < \exp(4|\lambda_0|e^x)$. The former implies 
	    $$|Df^{n_l}_\lambda(R(z,\lambda))|/|Df^{n_l}(z)| \approx 1,$$ 
	    for all $\lambda \in B$, which combined with \eqref{eq:dful} produces the bound  
	    $$|Df^{n_l}_\lambda(y)| < 2 e^{x^9} |\Re(f^{n_l}(z)) +2\pi|^4  < 3 e^{x^9} |\Re(f^{n_l}(z))|^4 $$
	    for $y \in  R_l(\widehat{U}_l, \lambda)$. 
	    As $f_\lambda^{n_l}$ maps $R_l(\widehat{U}_l, \lambda)$ biholomorphically onto $B(Q_l, 1)$, for each $z \in R_l(U_l, \lambda)$ there is a neighbourhood $V_z$ mapped biholomorphically by $f^{n_l}_\lambda$ onto $B(f^{n_l}_\lambda(z), 1) \subset B(Q_l, 1)$. By Lemma~\ref{lem:Mdev2}, say, the diameter of $V_z$ is bounded by $e^{-x}$. 

            From before, $f^j(U_l) \subset \cR(-2|\lambda_0|e^x -1)$ and the diameter of $f^j(U_l)$ is bounded by $e^{-x}$ for $j < n_l$. From \eqref{eq:Riterates}, $\dist(f_\lambda^j(z), f^j(U_l)) < e^{-x}$ for all $z \in R_l(U_l, \lambda)$. Thus $n_l$ is the first entry time for each point of $R(U_l, \lambda)$ (under iteration by $f_\lambda$) 
             to $\cL(-2|\lambda_0|e^x -2)$.
             Thus if $R(U_l,\lambda) \cap R(U_{l'},\lambda) \ne \emptyset$, $n_l = n_{l'}$, so $Q_l = Q_{l'}$ (as  $Q_l$ and $Q_{l'}$ either coincide or are disjoint), so $R(U_l,\lambda) = R(U_{l'}, \lambda)$. 
             In particular,
              the sets $R_l(U_l,\lambda)$, $1\leq l \leq L$, are pairwise-disjoint. 
            Define $R$ as the map whose restriction to each $U_l$ is $R_l$.

        It remains to show that  
        $R(U_l, B) \subset A \setminus B(\partial A, e^{-x})$. From above, $\dist(z, U_l) < e^{-x}$ for every $z \in R(U_l,\lambda)$ and each $\lambda \in B$,  
        and $U_l$ has diameter less than $e^{-x}$.
		Therefore $$\sup_{z',z \in U_l} \sup_{\lambda \in B} |R(z',\lambda) - z| < 2e^{-x}.$$ 
		Since there exists $z \in U_l \cap A_0$, so $z \in A \setminus B(\partial A, x^{-1/4})$, and $x^{-1/4} > 3e^{-x}$, we deduce that 
		$R(U_l, B) \subset A \setminus B(\partial A, e^{-x})$,
                as required.
            \end{proof}

   \section{Proof of Main Theorem} 
   The main theorem follows from the following proposition. The number $K$ is, we recall, the local degree of $h(0,\cdot)$ at $\lambda_0$, while $\xi_n(\lambda) = f^n_\lambda(0)$. We denote by $H$ the set of hyperbolic parameters.

   We shall use the estimates for passing from parameter to phase space near  $P(f)$ of Lemma~\ref{lem:xiinject},  and the estimates of Lemma~\ref{lem:overlaps} to go from near $P(f)$ to far out to the left. Their combination allows us to apply Lemma~\ref{lem:guide} to find large sets of hyperbolic parameters. 
   \begin{prop} \label{prop:main}
       Given $\eps >0$, there exists $\gamma, r_4 > 0$ such that for every $r \in (0, r_4)$ and every $4K$-sector $W$ of $A(\lambda_0; \gamma r,  r)$, 
       $$
       \frac{m(H\cap W)}{m(W)} > 1-\eps.
       $$
   \end{prop}
   \begin{proof}
       Let $ r_3, \gamma, \nu_0, C, C_0, C_1$ be given by Lemma~\ref{lem:xiinject}, for $\eps' = \eps/4$.
       For these $C, \nu_0$, let $M_2$ be given by Lemma~\ref{lem:overlaps}. 
       Let $C_2>0$ be large enough that
       $$
             C_1 \exp(11K |\lambda_0|e^x) 3 e^{x^9}  < 
             \exp(C_2 e^x) 
             $$
             for all $x > M_2$. 
             Let $M_3 > M_2$ be large enough that 
             \begin{itemize}
                 \item
                     $1/\log \log M_3 < \eps/3$;
                 \item
                     $M_3 > C_0,C_2$;
                 \item
                     Lemma~\ref{lem:guide} holds for the constant $C_2$ for all $x > M_3$;
                 \item
       $r_4 := \exp(-11|\lambda_0| e^{M_3}) < r_3.$
       \end{itemize}
       Let $r \in (0, r_4)$, so we can fix $x> M_3$ satisfying $r = \exp(-11|\lambda_0| e^x)$. 
         Let $n$ be given by Lemma~\ref{lem:xiinject}, let $W$ be as per the statement and set $A := \xi_n(W)$. 
         From Lemma~\ref{lem:xiinject}, $\xi_n$ is injective with distortion bounded by $1+\eps/4$ and $A$ is a simply-connected open set with $\partial A \leq C$.        Moreover $1/C_0r < |D\xi_n|$ on $W$.
	 The distortion bound implies 
	 \begin{equation}\label{eq:xiest}
	    |D\xi_n| < \sqrt{\frac{m(A)}{m(W)}}(1+\eps/4).
	    \end{equation}

         Meanwhile, $B(\lambda_0, r) \subset B := \{\lambda: |\log(\lambda_0/\lambda)| < \exp(-10|\lambda_0|e^x)\}$, so we can apply 
         Lemma~\ref{lem:overlaps}, 
         obtaining $R: \bigcup_{l=1}^L U_l \times B \to A \setminus B(\partial A, e^{-x})$ 
         together with the numbers $\{n_l\}_{l=1}^L$ and the estimates $|\log D_1R| < \exp(-e^x)|$ and $|D_2R| < \exp(4|\lambda_0| e^x)$. 
       
         Fix $l$ for now, and let $z \in U_l$. 
         Let 
         $$
         Y_z := \{R(z,\lambda) : \lambda \in B\} \subset A\setminus B(\partial A, e^{-x}).
         $$
         As $\exp(11 |\lambda_0| e^x)/C_0 =   1/C_0 r < |D\xi_n|$,
         \begin{equation}\label{eq:R2contract}
             |D_2R| < \exp(e^{-x}) |D\xi_n|.
         \end{equation}
         Hence
         the map $y \mapsto R(z, \xi_n^{-1}(y))$ is a strict contraction on $Y_z$ and  it has a unique fixed point $y_z \in \overline{Y_z} \subset A$. Let $\Lambda(z) := \xi_n^{-1}(y_z)$, 
         so $\xi_n(\Lambda(z)) = R(z, \Lambda(z))$. 
         
         Now $D\xi_n - D_2R \ne 0$, so we can apply the implicit function theorem to deduce that $z \mapsto \Lambda(z)$ is holomorphic  on each $U_l$. 
         Suppose $\Lambda(z)=\Lambda(z_1)$. From Lemma~\ref{lem:overlaps}, for 
         each $\lambda$, the sets $R(U_l, \lambda)$ are pairwise-disjoint, so $z$ and  $z_1$ must be in the same $U_l$. But on each $U_l\times \{\lambda\}$, $R$ is a homeomorphism, so $z = z_1$. Thus  $\Lambda(z)$ is injective on $U := \bigcup_{l=1}^L U_l$. 
         The map $\Lambda$ gives the link between parameter space and phase space.

        Taking derivative of $\xi_n(\Lambda(z)) = R(z, \Lambda(z))$ with respect to $z$, 
            $$
        D\xi_n(\Lambda(z))D\Lambda(z) = D_1R(z,\Lambda(z)) + D_2R(z,\Lambda(z)) D\Lambda(z),
            $$
            so
            \begin{equation}\label{eq:xilamb}
            D\Lambda(z) = \frac{D_1R(z,\Lambda(z))}{  - D_2R(z,\Lambda(z))+D\xi_n(z)}.
        \end{equation}
        Together with \eqref{eq:R2contract} and the estimate for  $|\log D_1R|$, \eqref{eq:xilamb} implies 
            $| D\Lambda(z)| > (1 - e^{-x})/ |D\xi_n(\Lambda(z))|$, say. 
	    Using \eqref{eq:xiest} and integrating $|D\Lambda|^2$ over $U$,  
	    $$
	    m( \Lambda(U)) > m(U) \frac{(1-e^{-x})^2}{(1+\eps/4)^2} \frac{m(W)}{m(A)}.
	    $$
         From Lemma~\ref{lem:overlaps} and choice of $M_3$, $m(U)/m(A) \geq 1 - 1/\log \log x > 1 - \eps/3$. 
             Thus 
             \begin{equation}\label{eq:UW1}
                 m(\Lambda(U))/m(W) \geq \frac{(1-e^{-x})^2}{(1+\eps/4)^2}(1-\eps/3) > 
                 1-\varepsilon.
             \end{equation}
             We have shown that $\Lambda(U)$ is a relatively large set. Next we show that it consists of hyperbolic parameters.

             Let $\lambda \in \Lambda(U_l)$ say and set $z := R(\Lambda^{-1}(\lambda), \lambda) = f^{n}_\lambda(0)$.  Let $V_z$ be given by Lemmma~\ref{lem:overlaps}, so $V_z$ of $z$ with diameter bounded by $e^{-x}$ gets mapped biholomorphically onto $B(f^{n_l}_\lambda(z), 1)$.
             For $j \leq n$, we know $f^j_\lambda(0) \in V$, so 
             by Lemma~\ref{lem:Vballs}, a neighbourhood of $0$ gets mapped biholomorphically onto $B(f^n_\lambda(0), \Delta \delta_0) \supset B(z, e^{-x})$. Therefore a neighbourhood of $0$ gets mapped biholomorphically by $f^{n + n_l}_\lambda$ onto 
             $$B(f^{n+n_l}_\lambda(0), 1) \subset \cL(-e^{x +\sqrt{x}} + 3\pi).$$ 
             From Lemma~\ref{lem:overlaps}, we have
             $$|Df_\lambda^{n_l}(z)| < 3e^{x^9} |\re(f^{n_l}(z))|^4,$$
             while Lemma~\ref{lem:xiinject} states that 
             $|Df_\lambda^n(0)| < C_1 /r^K.$ Recalling
             $r = \exp(-11|\lambda_0| e^x)$ and the choice of $C_2$,
             we obtain 
             $$|Df^{n + n_l}_\lambda(0)| < 
             C_1 \exp(11K |\lambda_0|e^x) 3 e^{x^9} |\re(f^{n+n_l}_\lambda(0))|^4 < 
             \exp(C_2 e^x) 
             |\re(f^{n+n_l}_\lambda(0))|^4. $$

             Applying Lemma~\ref{lem:guide}, $\lambda$ is a hyperbolic parameter.   This holds for each $\lambda \in \Lambda(U)$, so $\Lambda(U) \subset H$. Thus \eqref{eq:UW1} gives
            $$
            \frac{m(H\cap W)}{m(W)} \geq \frac{m(\Lambda(U))}{m(W)} \geq  
            1-\varepsilon,
            $$
            as required.
        \end{proof}
        The statement of the main theorem follows immediately from Proposition~\ref{prop:main}, so its proof is now complete.

            \section*{Acknowledgments}
            While this research was being conducted, the author was a Goldstine Fellow at the Business Analytics and Mathematical Sciences Department at the IBM T.\ J.\ Watson Research Center in New York. The author is very grateful to IBM and its staff for their support. The revision was carried out at University of Helsinki. The referee gave many helpful comments and also deserves thanks for suggesting proving non-existence of Lyapunov exponents. 

\bibliography{refs} 
\bibliographystyle{plain}
  
\end{document}